\documentclass[12pt]{amsart}  % AMS article class with 12pt font
\usepackage[utf8]{inputenc}   % UTF-8 input encoding for international characters

\usepackage{amsfonts}   % Additional math fonts (blackboard bold, etc.)
\usepackage{amsmath}    % Enhanced math environments and commands
\usepackage{amssymb}    % Extended mathematical symbols

\usepackage[hidelinks]{hyperref}   % Clickable links and cross-references in PDF

\usepackage{xcolor}     % Color support for text and backgrounds
\usepackage{soul}       % Text highlighting and letterspacing

\newtheorem{theorem}{Theorem}[section]              % Base counter
\newtheorem{proposition}[theorem]{Proposition}      % Shares counter with theorem
\newtheorem{lemma}[theorem]{Lemma}                  % Shares counter with theorem
\newtheorem{corollary}[theorem]{Corollary}          % Shares counter with theorem
\newtheorem{definition}[theorem]{Definition}        % Shares counter with theorem
\newtheorem{example}[theorem]{Example}              % Shares counter with theorem
\newtheorem{remark}[theorem]{Remark}                % Shares counter with theorem
\renewcommand{\hat}{\widehat}     % Make \hat produce wide hat by default

\providecommand{\tightlist}{%
  \setlength{\itemsep}{0pt}\setlength{\parskip}{0pt}}

\title{Asymptotic-Möbius maps}
\author[G. Grützner]{Georg Grützner}
\address{Department of Mathematics, University of Luxembourg, L-4364 Esch-sur-Alzette}
\email{georg.gruetzner@uni.lu}
\subjclass[2020]{Primary 30L10; Secondary 20F65, 51F30}
\keywords{asymptotic-Möbius maps, large-scale conformality, large-scale dimension, asymptotic cones, coarse cross-ratios, coarse geometry}
\begin{document}
\begin{abstract}
We introduce asymptotic-Möbius (AM) maps, a large-scale analogue of quasi-Möbius maps tailored to geometric group theory.
AM-maps capture coarse cross-ratio behavior for configurations of points that lie far apart, providing a notion of "conformality at infinity" that is stable under quasi-isometries, compatible with scaling limits, and rigid enough to yield structural consequences absent from Pansu's notion of large-scale conformality.
We establish basic properties of AM-maps, give several sources of examples, including quasi-isometries, sublinear bi-Lipschitz equivalences, snowflaking, and Assouad embeddings, and apply the theory to large-scale dimension and metric cotype.
As applications we obtain dimension-monotonicity results for nilpotent groups and CAT(0) spaces, and new obstructions to the existence of AM-maps arising from metric cotype.
\end{abstract}
\maketitle

\section{Introduction}
\label{sec:introduction}

Classical conformal geometry has no established ``rough'' analogue:
while quasiconformal, quasisymmetric, and quasi-Möbius mappings are well
understood in the infinitesimal or small-scale setting, there is no
coherent or generally accepted notion of ``conformality at large
scales'' within geometric group theory.

This paper aims to contribute to the development of such a framework by
identifying a class of mappings that behaves well under quasi-isometries
and whose scaling limits reflect classical quasisymmetric (QS), and
quasi-Möbius (QM) behavior on asymptotic cones.

In \cite{Pansu2021} Pierre Pansu introduces a notion of large-scale
conformal maps that mimics the infinitesimal behavior of conformal maps.
In short, large-scale conformal maps map families of disjoint balls onto
families of weakly overlapping quasiballs. It is a very flexible notion
that includes, for example, coarse embeddings. However, this flexibility
makes the asymptotic behavior of such maps less predictable.

This motivates the central construction of the present paper: the
introduction of \emph{asymptotic-Möbius (AM)} maps. Conceptually,
AM-maps are designed to mimic the defining feature of quasi-Möbius
maps---\emph{preservation of cross-ratios}---but only for configurations
of points that lie far apart. This gives a large-scale analogue of
quasi-Möbius behavior that is: stable under quasi-isometries and small
changes in topology, sensitive to the asymptotic geometry of the spaces
involved, and sufficiently rigid to derive structural consequences
absent from Pansu's framework.

We postpone the precise definition of asymptotic-Möbius (\(AM\)) maps
until later. Sources of examples of \(AM\)-maps are:

\begin{itemize}
\tightlist
\item
  Quasi-isometric embeddings,
\item
  Sublinear-bi-Lipschitz equivalences (i.e., maps inducing Lipschitz
  equivalences on asymptotic cones \cite{Cornulier2019},
\item
  Snowflaking (i.e., replacing a metric by a power of it),
\item
  Assouad maps from doubling metric spaces to \(\mathbb{R}^N\).
\end{itemize}

By Assouad's theorem \cite{Assouad1983}, every doubling metric
space---including every nilpotent Lie group and every finitely generated
group of polynomial growth---admits an AM-embedding into a sufficiently
high-dimensional Euclidean space. Moreover, we will encounter infinite
dimensional examples.

We apply the AM-framework to large-scale obstruction theory. Such
questions parallel classical quasiconformal obstruction theory but
operate at the level of large-scale cross-ratio geometry.

The aim of this paper is to develop the foundational theory of
asymptotic-Möbius maps, analyze their behavior under scaling limits, and
initiate a systematic approach to large-scale quasi-Möbius geometry
within geometric group theory.

Further related work on large-scale mappings has been carried out by
Cornulier \cite{Cornulier2011}, \cite{Cornulier2019} and by Pallier
\cite{Pallier2020}.

\subsection{Results}
\label{sec:results}

After introducing and developing the AM-framework in Section
\ref{sec:ammaps} and Section \ref{sec:examples}, we apply the
AM-framework to dimension theory and metric cotype.

In Section \ref{sec:applications_to_dimension_theory}, we show that
under AM-mappings a large-scale term of dimension increases. The
relevant term depends on the class of the groups considered.

\begin{theorem}
\label{0c2f8e}
Let \((\Gamma,|\cdot|)\) and \((\Gamma',|\cdot|)\) be finitely generated
nilpotent groups equipped with some word norm and
\(f: \Gamma \rightarrow \Gamma'\) a regular AM-map, then
\(\operatorname{asdim}(\Gamma) \leq \operatorname{asdim}(\Gamma')\). If
\(f\) is an AM-map between simply connected nilpotent Lie groups
\((G,d_{g})\) and \((G',d_{g'})\) equipped with left-invariant
Riemannian metrics, then
\(\operatorname{dim}(G) \leq \operatorname{dim}(G')\). Futhermore, if
\(\operatorname{asdim}(\Gamma) = \operatorname{asdim}(\Gamma')\), then
the asymptotic cones of \(\Gamma\) and \(\Gamma'\) are isomorphic graded
Lie groups. Conversely, given nilpotent groups with isomorphic
asymptotic cones, there exists an asymptotic-Möbius map between them.
\end{theorem}

In the world of \(\operatorname{CAT}(0)\)-spaces, the analogous theorem
takes the following form.

\begin{theorem}
\label{2f4976}
Let \(X\) and \(Y\) be extended metric spaces with points at infinity
\(*_{X}\) and \(*_{Y}\) respectively. Let \(X\) be asymptotically
chained, and assume that all asymptotic cones of \(X\) and \(Y\) are
\(\operatorname{CAT}(0)\)-spaces. If there exists a regular AM-map
\(f: (X,o, *_{X}) \rightarrow (Y,o',*_{Y})\), then the telescopic
dimension increases e.g.
\[\operatorname{tele-dim}(X) \leq \operatorname{tele-dim}(Y).\]
\end{theorem}

In Section \ref{sec:on_metric_cotype_obstructions} we apply
Mendel--Naor's metric cotype theory to derive additional obstructions to
the existence of AM-maps.

\begin{theorem}
\label{7f9774}
Suppose that \(p,q \in [2,\infty]\) satisfy \(p > q\). Then there does
not exist a regular AM-map from \(\ell_{p}\) to a metric space with
sharp metric cotype \(q\).
\end{theorem}

The paper is organized as follows. In Section \ref{sec:ammaps} we
introduce and develop the general AM-framework. In Section
\ref{sec:examples} we illustrate the theory: first, we show that every
sublinear-Lipschitz equivalence is an AM-map (Proposition \ref{7b5b3e});
then we construct an Assouad-type map between infinite-dimensional
spaces as a further example of an AM-map (Proposition \ref{49a80e}). In
Section \ref{sec:applications_to_dimension_theory} we apply the
AM-framework to dimension theory and prove Theorem \ref{0c2f8e} and
Theorem \ref{2f4976}. In Section \ref{sec:on_metric_cotype_obstructions}
we apply Mendel--Naor's metric cotype to the AM-framework and prove
Theorem \ref{7f9774}.

I would like to thank Pierre Pansu for introducing me to this topic and
for pointing me toward several of its applications.

\section{AM-maps}
\label{sec:ammaps}

\subsection{Extended metrics}
\label{sec:extended_metrics}

The metric distortion of conformal mappings can be substantial. For
example, the stereographic projection exhibits a rapidly increasing
metric distortion (Lipschitz constant) near the pole \(o\). Nonetheless,
its conformal behavior remains well-behaved all the way up to \(o\).

To incorporate the pole \(o\) in our topological model while still
allowing for infinite metric distortion, we extend the notion of a
metric.

\begin{definition}
Given a set \(X\), an \emph{extended metric} on \(X\) is a function
\(d: X \times X \rightarrow \mathbb{R} \cup \{\infty\}\) with the
following properties:

\begin{enumerate}
\def\labelenumi{\arabic{enumi}.}
\tightlist
\item
  There exists at most one point \(*_{X} \in X\) such that:
\end{enumerate}

\begin{itemize}
\tightlist
\item
  For all \(x \in X \backslash\{*_{X}\}\),
  \(d_{*_{X}}(x, *_{X}) = \infty\)
\item
  \(d(*_{X}, *_{X}) = 0\)
\end{itemize}

\begin{enumerate}
\def\labelenumi{\arabic{enumi}.}
\setcounter{enumi}{1}
\tightlist
\item
  The restriction of \(d\) to
  \((X \backslash\{*_{X}\}) \times (X \backslash\{*_{X}\})\) is a
  metric.
\end{enumerate}

The point \(*_{X}\) is referred to as the \emph{point at infinity} of
\((X,d)\) or simply of \(d\) and points in \(X_{*}\) are referred to as
\emph{finite points}.
\end{definition}

For better readability, we denote \(X \backslash\{*_{X}\}\) by
\(X_{*_{X}}\). We also sometimes write \(d_{*_{X}}\) for \(d\) in
analogy of the metric Cayley transform of \(d\) at a finite point \(p\).

\begin{remark}
\label{7ac1d2}
An extended metric takes values in the projectively extended real line
in which division by zero is allowed. In particular, \(1 / 0=\infty\)
and \(1 / \infty=0\). There is however no distinction between
\(+\infty\) and \(-\infty\).
\end{remark}

\begin{remark}[extended metric topology{}]
\label{f3fa12}
The \emph{extended metric topology} on an extended metric space \(X\) is
the topology associated with the base consisting of

\begin{itemize}
\tightlist
\item
  all open ball \(B(x,r)\subset X_{*_{X}}\), and
\item
  sets of the form
  \[U(x,r) \;=\;\{*_{X}\}\ \cup\ \Big(X\setminus \overline{B(x,r)}\Big)\]
  for any \(x \in X_{*_{X}}\) and any \(r>0\), if \(X\) admits a point
  at infinity.
\end{itemize}
\end{remark}

\begin{remark}
\label{8eb339}
The extended metric topology agrees with Buyalo's semimetric topology
\(\mathcal{T}_{\rho}\) \cite{Buyalo2016}, when \(\rho\) is an extended
metric.
\end{remark}

\begin{remark}
\label{aa3217}
If \((X_{*},d)\) is a proper metric space, then \(X\) equipped with the
extended metric topology is exactly the Alexandroff extension of
\(X_{*}\). In particular, \((X,d)\) is compact in this case.
\end{remark}

\begin{remark}
\label{fea4a2}
The extended metric topology has the advantage of being metrizable even
if \((X_{*},d)\) is not locally compact. If \((X_{*},d)\) is unbounded
(not necessarily locally compact), then the extended metric topology on
\(X\) is induced by the inverse metric Cayley transform of \(d\) (see
Definition \ref{226dfd} and \cite[Lemma~2.2]{Bonk2002}).
\end{remark}

\begin{remark}
\label{e6e25e}
In light of Remark \ref{fea4a2}, if an extended metric space \((X,d)\)
admits a point at infinity \(*_{X}\), then we always assume
\((X_{*_{X}},d)\) to be unbounded.
\end{remark}

\subsubsection{Quasimetrics}
\label{sec:quasimetrics}

``Inverting a metric'' at a point is an archetypal operation that is
conformal but exhibits unbounded distortion. It turns out that the
triangle inequality is slightly too strong a requirement to be preserved
under such a metric inversion.

\begin{definition}
A \emph{\(K\)-quasimetric} is a symmetric function
\(d: X \times X \to \mathbb{R} \cup \{ \infty \}\) such that there
exists \(K > 0\) and an extended metric \(\rho\) on \(X\), s.t.
\[\frac{1}{K} \rho(x,y) \leq d(x,y) \leq K \rho(x,y).\] A
\emph{quasimetric space} is a set \(X\) together with a quasimetric
\(d\).
\end{definition}

\begin{remark}
\label{56ca12}
A quasimetric \(d\) satisfies the triangle inequality just enough, such
that open balls defined via \(d\) form a basis of a topology. The
extended metric topology on \(X\) coming from \(d\) is the same as the
extended metric topology on \(X\) coming from \(\rho\)
\cite[Lemma~13.3]{Munkres2000}.
\end{remark}

\begin{definition}
Let \((X,d)\) be an quasimetric space. Let \(x, y, z, w\) be points in
\(X\) with \(x \neq w\) and \(y \neq z\). If all points are finite,
their \emph{cross ratio} \([x, y, z, w]_{d}\) is defined by
\[[x, y, z, w]_{d}=\frac{d(x,z)d(y,w)}{d(x,w)d(y,z)}.\] If \(x=z\) or
\(y=w\), we set \([x, y, z, w]_{d}=0\) even if some of these points are
the point at infinity \(*_{X}\). In other cases we omit the factors
containing \(*_{X}\). For example,
\[[x, y, z, *_{X}]_{d}=\frac{d(x,z)}{d(y,z)}.\]
\end{definition}

The notion of a cross ratio originated in projective geometry of the
19th century and appeared in the work of Möbius, Laguerre, and later in
conformal geometry. Its extension to general metric spaces was
systematically developed in the theory of quasiconformal and
quasiregular maps in the papers of Väisälä \cite{Vaisala1984}.

\begin{definition}
Let \(X\), \(Y\) be quasimetric spaces. Suppose that \(A \subset X\) and
that \(f: A \rightarrow Y\) is an embedding. We say that \(f\) is
\emph{Möbius} if \[[f(x),f(y),f(z),f(w)] = [x,y,z,w]\] whenever
\([x,y,z,w]\) is a cross ratio of points in \(A\). Two quasimetric
spaces are Möbius equivalent if there exists a Möbius homeomorphism
\(f: X \to Y\), \(f\) is then called a \emph{Möbius equivalence}. Two
quasimetrics \(d\) and \(d'\) on \(X\) are Möbius equivalent, if the
identity map \(id: (X,d) \to (X,d')\) is a Möbius equivalence.
\end{definition}

\begin{remark}
\label{589a02}
Two Möbius equivalent quasimetrics induce the same extended metric
topology, see \cite[Proposition~4.4]{Buyalo2016} and Remark
\ref{8eb339}.
\end{remark}

\begin{definition}
\label{7c701f}
Let \((X,d)\) be a quasimetric space. The \emph{Cayley transform of
\((X,d)\) at a finite point \(p \in X\)} is the space \((X,d_{p})\),
where \(d_{p}\) is defined by
\[d_p(x,y) = \frac{d(x,y)}{d(x,p)\,d(p,y)} \quad \text{if} \; x,y \in X \setminus \{p\},\]
\(d_p(x,p) = \infty\) for all \(x \in X \setminus \{p\}\) and
\(d_p(p,p) = 0\). If \((X,d)\) already has a point at infinity
\(*_{X}\), then the convention is that
\(d_p(x,*_{X}) = \frac{1}{d(x,p)}\).

Formally the \emph{Cayley transform of \((X,d)\) at a finite point
\(p \in X\)} is the identity map
\(\mathcal{C}_{p}:(X,d) \to (X,d_{p})\).
\end{definition}

\begin{remark}
\label{47aba5}
The Cayley transform preserves the metric cross ratio and is thus a
homeomorphism as of Remark \ref{589a02}. In other words, the Cayley
transform is a Möbius equivalence.
\end{remark}

\begin{example}
\label{ff036e}
In \(\mathbb{R}^n\) with the Euclidean metric, the Cayley transform at
\(p=0\) produces \[d_0(x,y) = \frac{\|x-y\|}{\|x\|\,\|y\|},\] the
familiar chordal metric on \(S^{n}\) minus one point, up to Möbius
equivalence.
\end{example}

\begin{definition}
\label{226dfd}
Let \((X,d)\) be a quasimetric space with a point at infinity \(*_{X}\).
The inverse Cayley transform of \((X,d)\) at a finite point \(q \in X\)
is the space \((X,d^{q})\), where \(d^{q}\) is defined by
\[d^{q}(x,y) = \frac{d(x,y)}{(1 + d(x,q))(1 + d(y,q))} \quad \text{if} \; x,y \in X \setminus \{*_{X}\},\]
\(d^{q}(x,*_{X}) = d^{q}(*_{X},x) = \frac{1}{1 + d(x,q)}\) for all
\(x \in X \setminus \{*_{X}\}\), and \(d^{q}(*_{X},*_{X}) = 0\).

Formally the \emph{inverse Cayley transform of \((X,d)\) at a finite
point \(q \in X\)} is the identity map
\(\mathcal{C}^{p}:(X,d) \to (X,d^{q})\).
\end{definition}

\begin{remark}
\label{034f19}
The metric Cayley transform and the inverse metric Cayley transform map
quasimetrics to quasimetrics \cite{Bonk2002}. This is generally false
for extended metrics and the main reason for introducing the class of
quasimetrics in the first place.
\end{remark}

The following remark shows, however, that one could dispense with the
introduction of quasimetrics at the cost of systematically applying a
smoothing procedure. We have chosen not to pursue this approach in the
present paper.

\begin{remark}
\label{b38b60}
There is a standard procedure to pass from a quasimetric \(d\) to a
bi-Lipschitz equivalent extended metric \(\hat{d}\) by defining
\[\hat{d}(x, y):=\inf \sum_{i=0}^{k-1} d\left(x_i, x_{i+1}\right),\]
where the infimum is taken over all finite sequences of points
\(x_0, \ldots, x_k \in X\) with \(x_0=x\) and \(x_k=y\)
\cite[Lemma~2.2]{Bonk2002}.
\end{remark}

\subsection{Quasi-Möbius maps}
\label{sec:quasimobius_maps}

Quasi-Möbius maps were introduced by Jussi Väisälä, among others, as a
means of studying quasisymmetric maps and quasiconformal maps. Unlike
quasisymmetric maps, quasi-Möbius maps do not have a point fixed at
infinity.

\begin{definition}
Let \(X\), \(Y\) be quasimetric spaces. Suppose that \(A \subset X\) and
that \(f: A \rightarrow Y\) is an embedding. We say that \(f\) is
\emph{quasi-Möbius} or \(\mathrm{QM}\) if there is a homeomorphism
\(\eta:[0, \infty) \rightarrow[0, \infty)\) such that
\[[f(x),f(y),f(z),f(w)] \leq \eta([x,y,z,w])\] whenever \([x,y,z,w]\) is
a cross ratio of points in \(A\).
\end{definition}

Examples of quasi-Möbius maps are:

\begin{itemize}
\tightlist
\item
  the stereographic projection
  \(S^n \setminus \{o\}\rightarrow \mathbb{R}^n\),
\item
  the Cayley transformations (the complex, quaternionic and octonionic
  analogues of the stereographic projection) \cite{Astengo2004},
\item
  the inversions \(x \mapsto \frac{x}{|x|^2}\) in Banach spaces
  \cite{Vaisala1984}.
\end{itemize}

\begin{remark}
\label{59a2c8}
Quasi-Möbius maps are an intermediate class between quasisymmetric and
quasiconformal maps. Every quasisymmetric map is quasi-Möbius, every
quasi-Möbius map is quasiconformal. The converse is true in special
cases. A quasi-Möbius map between bounded spaces is quasisymmetric. A
quasi-Möbius map between unbounded spaces is quasisymmetric if and only
if it is proper. For details see \cite{Vaisala1984}.
\end{remark}

\subsection{Asymptotic Möbius maps}
\label{sec:asymptotic_mobius_maps}

In the following we will work in the category of doubly pointed
quasimetric spaces.

\begin{definition}
A \emph{doubly pointed space} is a set \(X\) with two distinct points
\(\{ q,p \} \subset X\). We denote doubly pointed spaces by \((X,q,p)\).
Morphisms \((X,q,p) \to (Y,q',p')\) are functions \(f: X \to Y\) s.t
\(f(q) = q'\) and \(f(p) = p'\).
\end{definition}

We refer to \(q\) as the \emph{origin} of \(X\) and to \(p\) as the
\emph{antipodal point} of \(X\) w.r.t the origin \(q\).

By convention, \((X,q,p)\) is either a bounded quasimetric space or an
unbounded quasimetric space equipped with a point at infinity. In the
unbounded case, we require that either \(q\) or \(p\) is the point at
infinity; typically, we choose \(p\) to play this role.

We further assume that the antipodal point \(p\) is always an
accumulation point of \(X\). This condition is automatic when \(p\) is
the point at infinity, since in that case the underlying metric space
\((X_{*},d)\) is assumed to be unbounded.

We need a criterion to tell when two points in a space are far apart.
One way to do this is to separate them by sublinear growing functions.
This idea was used by Y. Cornulier to define sublinear-bi-Lipschitz
equivalences \cite{Cornulier2019}.

Let \((X,d,q,p)\) is a doubly pointed quasimetric space. Given a
function \(u: \mathbb{R}_+ \to \mathbb{R}\), we say that two points
\(x, y \in X_{p}\) are \emph{separated by \(u\) at \(p\)} if \[
d_{p}(x,y) > u\big(d_{p}(x,q) + d_{p}(q,y)\big).
\] By convention, any point \(x \in X_{p}\) is separated from \(p\) at
\(p\) for any function \(u\). Here \(d_{p}\) denotes the metric Cayley
transform of \(d\) if \(p\) is a finite point. Otherwise it denotes
\(d\) itself.

For simplicity, we often write \(|x|_{p} = d_{p}(o,x)\) for the distance
from the origin, and \(x,y >_{p} u\) for \(x\) and \(y\) being separated
by \(u\) at \(p\). When the point \(p\) is clear from the context, we
may also write \(|x|\) for the distance to the origin and \(x,y > u\)
for the separation by \(u\).

We can call \(u\) a \emph{gauge}, i.e., a function that sets a scale of
``sight'' as a function of ``location''. To fix a gauge means to decide
what is ``near'' and ``far away''.

\begin{definition}
A function \(u: \mathbb{R}_+ \rightarrow \mathbb{R}\) is an admissible
gauge if

\begin{itemize}
\tightlist
\item
  it is nondecreasing, and
\item
  \(u\) grows sublinearly,
  i.e.~\(\limsup_{r \rightarrow \infty}\frac{u(r)}{r} = 0\).
\end{itemize}
\end{definition}

\begin{example}
\label{41fd96}
Take \(u\) to be \(u(r) = \log(n + r)\) where \(n \in \mathbb{N}\).
\end{example}

\begin{example}
\label{a516c1}
Let \((X,d,q, *_{X})\) be an unbounded doubly pointed quasimetric space
and \(u\) an admissible gauge. Any sequence \(\{ x_{n} \}\), s.t.
\(\lim_{n \to \infty} d(q,x_{n}) = \infty\) is separated by \(u\) from
the constant sequence \(\{ q \}\).
\end{example}

\begin{definition}
\label{f87e84}
Let \(f: (X, q, p) \to (Y, q', p')\) be a map between doubly pointed
quasimetric spaces. We say that \(f\) is an \emph{asymptotic-Möbius map
at \(p\)} (or an \emph{AM-map at} \(p\)) if the following condition
holds:

\begin{itemize}
\tightlist
\item
  There exists an admissible gauge \(u\) and a homeomorphism
  \(\eta: \mathbb{R}^+ \to \mathbb{R}^+\) such that for every quadruple
  \(x, y, z, w \in X\) that are pairwise separated by \(u\) at \(p\),
  the following inequality holds:
  \[[f(x), f(y), f(z), f(w)] < \eta([x, y, z, w]).\]
\end{itemize}

If we want to emphasize the gauge \(u\), we say that \(f\) is an
asymptotic-Möbius map at scale \(u\) at \(p\), and write \(f_{u}\). In
the special case when \(p\) is the point at infinity, \(f\) is called an
\emph{asymptotic-Möbius map at infinity}.
\end{definition}

Morally, asymptotic-Möbius maps at widely spaced points are
quasi-Möbius.

\begin{definition}
An asymptotic-Möbius map \(f: (X, q, p) \to (Y, q', p')\) at \(p\)
between doubly pointed quasimetric spaces is called \emph{regular}, if

\begin{enumerate}
\def\labelenumi{\arabic{enumi}.}
\tightlist
\item
  \(f\) is \emph{continuous} at \(p\).
\item
  For every neighborhood \(V\) of \(p\), there exists a neighborhood
  \(U\) of \(p'\) such that \[f^{-1}(U) \subset V.\]
\end{enumerate}
\end{definition}

\begin{example}
\label{a67dd2}
The Cayley transform and inverse metric Cayley transform are Möbius
equivalences, in particular they are regular AM-maps at all points and
all scales.
\end{example}

At some point we will need a coarse version of path connectivity.

\begin{definition}
A doubly pointed quasimetric space \((X,q,p)\) is \emph{asymptotically
chained at \(p\)}, if there exists an admissible gauge \(v_{q}\) such
that for all \(x,y \in X_{p}\) there exists a chain
\(x_1 = x, \dots, x_{k + 1} = y\) satisfying
\[\operatorname{max}_{\,i \in 1 \dots k}\{ d_{p}(x_i,x_{i+1})\} < v_{q}(|x|_{p} + |y|_{p}).\]
Here \(d_{p}\) denotes the Cayley transform of \(d\) if \(p\) is a
finite point, otherwise it denotes simply the quasimetric \(d\) itself.
\end{definition}

We say that \((X,q,p)\) is asymptotically chained at \(p\) w.r.t.
\(v_{q}\). If \(p\) is the point at infinity \(*_{X}\), we simply say
that \((X,q,*_{X})\) is \emph{asymptotically chained}.

\begin{remark}
\label{977137}
If \((X,q,*_{X})\) is asymptotically chained, then any asymptotic cone
of \(X\) at \(q\) is isometric to an asymptotic cone
\(\operatorname{Cone}^\omega_{\lambda_{n}} (X,q)\), where
\(\lambda_{n} = d(q,z_{n})\) for some sequence \(z_{n} \in X\), see
Lemma \ref{0c35a0}.
\end{remark}

Our most important technical step will be the following theorem, the
proof of which will occupy Section \ref{sec:amtheorem} after some
introductory definitions.

\begin{theorem}[AM{}]
\label{85f2fe}
Let \((X,q,p)\) be a doubly pointed quasimetric space that is
asymptotically chained at \(p\) and \(Y\) a quasimetric space. Let
\(f: (X,q,p) \to (Y,q',p')\) be a regular AM-map at \(p \in X\).

\begin{enumerate}
\def\labelenumi{\arabic{enumi}.}
\tightlist
\item
  If \(p\) is a finite point and \(f\) maps \(p\) to a finite point in
  \(Y\), then \(f\) induces a continuous, injective, quasimöbius map
  \(g\) between some tangent cones of \(X\) and \(Y\).
\item
  If \(p\) is a finite point and \(f\) maps \(p\) to the point at
  infinity of \(Y\) (if it exists), then \(f\) induces a continuous,
  injective, quasimöbius map \(g\) between some tangent cone of \(X\)
  and some asymptotic cone of \(Y\).
\item
  If \(p\) is the point at infinity of \(X\) (if it exists) and \(f\)
  maps \(p\) to a finite point in \(Y\), then \(f\) induces a
  continuous, injective, quasimöbius map \(g\) between some asymptotic
  cone of \(X\) and some tangent cone of \(Y\).
\item
  If \(p\) is the point at infinity of \(X\) (if it exists) and \(f\)
  maps \(p\) to the point at infinity of \(Y\) (if it exists), then
  \(f\) induces a continuous, injective, quasimöbius map \(g\) between
  some asymptotic cones of \(X\) and \(Y\).
\end{enumerate}
\end{theorem}

\subsection{Asymptotic cones and tangent cones}
\label{sec:asymptotic_cones_and_tangent_cones}

The \emph{asymptotic cone} of a metric space \((X,d)\), captures the
geometry on the large scale of \(X\). Roughly speaking, it formalizes
the idea of snapshots of the space \(X\) taken by an observer moving
farther and farther away from \(X\). This sequence of snapshots can
stabilize and the observer has the impression of seeing a single object.
We call this object the asymptotic cone of \(X\).

If \(R\)-balls \(B(o,R)\) with metric \(R^{-1} d\) are uniformly
pre-compact, then the asymptotic cones can be constructed concretely by
pointed Gromov Hausdorff convergence \cite{Gromov1981},
\cite{Gromov2007}. In general, however, we have to resort to
constructions via ultrafilters.

\subsubsection{Nonprinciple ultrafilters and ultrapowers }
\label{sec:nonprinciple_ultrafilters_and_ultrapowers}

\begin{definition}[ultrafilter{}]
A (nonprincipal) ultrafilter \(\omega\) over \(\mathbb{N}\) is a set of
subsets of \(\mathbb{N}\) satisfying the following conditions:

\begin{itemize}
\tightlist
\item
  If \(A,B \in \omega\) then \(A \cap B \in \omega\).
\item
  If \(A \in \omega, A \subset B \subset \mathbb{N}\), then
  \(B \in \omega\).
\item
  For every \(A \subset \mathbb{N}\), either \(A \in \omega\) or
  \(\mathbb{N}\setminus A \in \omega\).
\item
  No finite subset of \(\mathbb{N}\) is in \(\omega\).
\end{itemize}

Equivalently, \(\omega\) is a finitely additive probability measure on
\(\mathbb{N}\) such that every subset has measure either \(0\) or \(1\)
and every finite subset has measure \(0\).
\end{definition}

If a proposition \(P(n)\) holds for all \(n \in A\), where \(A\) belongs
to an ultrafilter \(\omega\), then \(P(n)\) is said to hold
\emph{\(\omega\)-almost surely}.

\begin{definition}[$\omega$-limit{}]
Let \(\omega\) be a (nonprincipal) ultrafilter over \(\mathbb{N}\). An
\(\omega\)-limit of a sequence of points \(\{x_n\}\) in a topological
space \(X\) is a point \(x\) in \(X\) such that for any neighborhood
\(U\) of \(x\) the relation \(x_n \in U\) holds \(\omega\)-almost
surely.
\end{definition}

If \(X\) is a Hausdorff space, then the \(\omega\)-limit of a sequence
is unique. We denote this point by \(\lim_\omega x_n\).

\begin{remark}
\label{211a51}
Any sequence in a compact~Hausdorff first-countable topological space
has an \(\omega\)-limit.

If \(X\) is a complete metric space that is not locally compact, even a
bounded sequence may not have an \(\omega\)-limit. Take for example
\(X=\ell^2\) and let \(e_n\) be the standard orthonormal basis. Consider
the sequence \(\{ e_{n} \}\). For any fixed \(x\in\ell^2\) the
coordinates \(x_n\) of \(x\) tend to \(0\), so
\[\|e_n-x\|^2 = 1+\|x\|^2-2\langle e_n,x\rangle \longrightarrow 1+\|x\|^2,\]
hence \(\|e_n-x\|\) stays bounded away from \(0\). Thus there exists a
neighborhood \(U\) of \(x\) such that the set \(\{n:\,e_n\in U\}\) is
finite (hence not in \(\omega\)). Therefore \(\{ e_{n} \}\) has no
\(\omega\)-limit in \(\ell^2\).
\end{remark}

\begin{definition}[ultrapower{}]
\label{e54e25}
The ultrapower of a set \(X\) with respect to an ultrafilter \(\omega\),
denoted by \(X^\omega\), consists of equivalence classes of sequences
\(\{x_n\}\), \(x_n \in X\), where two sequences \(\{x_n\}\) and
\(\{y_n\}\) are identical if and only if \(x_n = y_n\) \(\omega\)-almost
surely.
\end{definition}

We adopt the notation \(\{x_n\}_\omega\) for the equivalence class of
sequences.

\begin{example}
\label{072303}
The ultrapower of \(\ell_{p}\) is linearly isometric to some
\(L^p(\mu)\) \cite{Dacunha-Castelle1972}.
\end{example}

\begin{remark}
\label{5ff9c9}
If \(X\) is a metric space, then for any sequence \(\{ x_{n} \}\) in
\(X\), either \(\lim_{\omega} d(o,x_{n}) = L < \infty\) or
\(\lim_{ \omega  } d(o,x_{n}) = \infty\) i.e.~for any \(C \geq 0\),
\(d(o,x_{n}) > C\), \(\omega\)-almost surely. Indeed, for any \(C > 0\),
and \[A_{C} = \{ n: d(o,x_{n}) \leq C \}\] either \(A_{C} \in \omega\)
or \(\mathbb{N} \setminus A_{C} \in \omega\). If \(A_{C} \in \omega\)
for some \(C\), then the \(\omega\)-limit of \(d(o,x_{n})\) exists and
is finite.
\end{remark}

\subsubsection{Asymptotic cones}
\label{sec:asymptotic_cones}

Let \((X,d)\) be a quasimetric space and \(\omega\) an ultrafilter over
\(\mathbb{N}\). Furthermore, let a sequence of numbers \(\{\lambda_n\}\)
be given, with
\href{limit_ultrafilters}{\(\lim_\omega \lambda_n = \infty\)}. Then
there is a pseudo-quasimetric on the ultrapower \(X^\omega\) by setting
the distance between \(x = \{x_n\}_\omega\) and \(y = \{y_n\}_\omega\)
to
\[D_{d/\lambda_{n}}(x,y) = \lim_\omega \frac{d(x_n,y_n)}{\lambda_n}.\]
Fix a finite point \(q \in X\) as origin and let
\(\operatorname{Cone}^\omega_{\lambda_{n}} (X,q)\) be the set of
equivalence classes of elements \(x = \{x_n\}_\omega\) in \(X^\omega\)
w.r.t \(D_{d/\lambda_{n}}\), i.e.~two elements \(x = \{x_n\}_\omega\)
and \(y = \{y_n\}_\omega\) of \(X^\omega\) s.t.
\(\lim_{ \omega } \frac{d(q,x_{n})}{\lambda_{n}} < \infty\) are
identical if and only if \(D_{d/\lambda_{n}}(x,y) = 0\). The point at
infinity \(*\) is represented by all sequences \(\{ x_{n} \}_{\omega}\)
s.t. \(\lim_{ \omega } \frac{d(q,x_{n})}{\lambda_{n}} = \infty\).

The quasimetric space
\((\operatorname{Cone}^\omega_{\lambda_{n}} (X,q),D_{d/\lambda_{n}})\)
is called an \emph{asymptotic cone} of \((X,d)\). We denote points in
the asymptotic cone
\((\operatorname{Cone}^\omega_{\lambda_{n}} (X, q),D_{d/\lambda_{n}})\)
by \(x_\omega\) or \([x_{n}]_{\omega}\). The fixed origin \(q \in X\)
gives rise to an origin \(q_{\omega}\) in the asymptotic cone
\((\operatorname{Cone}^\omega_{\lambda_{n}} (X, q),D_{d/\lambda_{n}})\).

\begin{remark}[{\cite[Lemma~5.53]{Bridson1999}}]
\label{5457e8}
Any asymptotic cone of a given metric space \(X\) is complete.
\end{remark}

\begin{remark}
\label{2c7630}
If \(X\) is a Banach space, then all asymptotic cones
\(\operatorname{Cone}^\omega_{\lambda_{n}} (X,0)\) of \(X\) are linearly
isometric to the ultrapower \(X^{\omega}\) of \(X\), see Lemma
\ref{ae93d0}.
\end{remark}

\begin{remark}[{\cite[Corollary~3.10~(4)]{Bridson1999}}]
\label{5dd0f4}
Any asymptotic cone of a \(\operatorname{CAT}(0)\)-space is a
\(\operatorname{CAT}(0)\)-space.
\end{remark}

\begin{example}
\label{1554d9}
From Remark \ref{2c7630} and Example \ref{072303}, it follows that all
asymptotic cones of \(\ell_{p}\) are linearly isometric to some
\(L_{p}(\mu)\).
\end{example}

\begin{remark}
\label{3620f5}
One may take any point \(x = \{ x_{n} \}_{\omega}\) in the ultrapower
\(X^{\omega}\) as origin. It is easy to verify that for any two points
\(x, y \in X^{\omega}\) with \(D(x,y)< \infty\),
\[\operatorname{Cone}^\omega_{\lambda_{n}} (X, x) = \operatorname{Cone}^\omega_{\lambda_{n}} (X, y).\]
\end{remark}

\begin{remark}
\label{374bcc}
For any sequence \(\{ x_{n} \}\) representing a finite point
\(x_{\omega}\) in the asymptotic cone
\(\operatorname{Cone}^{\omega}_{\lambda_{n}}(X,d,q)\). There exists an
\(M \geq 0\) s.t. \[d(q,x_{n}) \leq M\lambda_{n}\] for \(\omega\)-almost
every \(n\). This follows from the fact that by assumption, the sequence
\(\frac{d(q,x_{n})}{\lambda_{n}}\) is \(\omega\)-a.s. bounded.
\end{remark}

Recall that throughout this paper we work either with a bounded
quasimetric space or with a quasimetric space \((X,d)\) equipped with a
point at infinity \(*\), such that the underlying space \((X_{*},d)\) is
unbounded. Since every asymptotic cone of a bounded quasimetric space
reduces to a single point, our attention is restricted to asymptotic
cones of the latter type.

\begin{remark}
\label{fa1dba}
Notice that any point
\(x_{\omega} \in \operatorname{Cone}^{\omega}(X;o,\lambda_{n})\) that is
different from the origin \(o_{\omega}\) is associated with a sequence
\(\{ x_{n} \}\) in \((X_{*},d_{*})\) that \(\omega\)-a.s. converges to
the point \(*\), i.e.~\(\lim_{ \omega} x_{n} = *\).

This is clear if \(X_{*}\) is proper, since then \(X\) is compact
(Remark \ref{aa3217}) and \(\omega\)-limits always exist. Since
\(\lim_{ \omega  }d_{*}(o,x_{n}) = \infty\), the sequence necessarily
converges to \(*\). If \(X_{*}\) is not proper, then convergence follows
from the fact that \(\lim_{ \omega }d^{q}(x_{n},*) = 0\) for the inverse
metric Cayley transform \(d^{q}\) of \(d\), since
\(d^{q}(x_{n},*) \sim \frac{1}{1 + d(q,x_{n})}\) (see Remark
\ref{fea4a2}).
\end{remark}

\subsubsection{Tangent cones}
\label{sec:tangent_cones}

Let \((X,d)\) be a quasimetric space and \(\omega\) an ultrafilter over
\(\mathbb{N}\). Furthermore, let a sequence of numbers \(\{\lambda_n\}\)
be given, with \(\lim_\omega \lambda_n = \infty\). Then there is a
pseudo-quasimetric on the ultrapower \(X^\omega\) by setting the
distance between \(x = \{x_n\}_\omega\) and \(y = \{y_n\}_\omega\) to
\[D_{\lambda_{n}d}(x,y) = \lim_\omega \lambda_{n} \, d(x_n,y_n).\] Fix a
finite point \(p \in X\) as origin and let
\(\operatorname{Tan}^\omega_{\lambda_{n}} (X,p)\) be the set of
equivalence classes of elements \(x = \{x_n\}_\omega\) in \(X^\omega\)
w.r.t. \(D_{\lambda_{n}d}\), i.e.~two elements \(x = \{x_n\}_\omega\)
and \(y = \{y_n\}_\omega\) of \(X^\omega\) satisfying
\(\lim_\omega \lambda_{n} \, d(x_n,p) <\infty\) are identical if and
only if \(D_{\lambda_{n}d}(x,y) = 0\). The point at infinity \(*\) is
represented by all sequences \(\{ x_{n} \}_{\omega}\) s.t.
\(\lim_{ \omega } \lambda_{n}\,d(p,x_{n}) = \infty\).

The quasimetric space
\((\operatorname{Tan}^\omega_{\lambda_{n}} (X,p),D_{\lambda_{n}d})\) is
called a \emph{tangent cone} of \((X,d)\) at \(p\). We denote points in
the tangent cone
\((\operatorname{Tan}^\omega_{\lambda_{n}} (X, p),D_{\lambda_{n}d})\) by
\(x_\omega\). The fixed origin \(p \in X\) gives rise to an origin
\(p_{\omega}\) in the tangent cone
\((\operatorname{Tan}^\omega_{\lambda_{n}} (X, p),D_{\lambda_{n}d})\).

\begin{remark}
\label{81d7f8}
For any sequence \(\{ x_{n} \} \subset (X,d)\) representing a finite
point \(x_{\omega}\) in the tangent cone
\(\operatorname{Tan}^{\omega}_{\lambda_{n}}(X,p)\). There exists an
\(M \geq 0\) s.t. \[d(p,x_{n}) \leq \frac{M}{\lambda_{n}}\] for
\(\omega\)-almost every \(n\). This follows from the fact that by
assumption, the sequence \(\lambda_{n}\,d(p,x_{n})\) is \(\omega\)-a.s.
bounded.
\end{remark}

\begin{remark}
\label{e29e4b}
Any point
\(x_{\omega} \in \operatorname{Tan}^{\omega}_{\lambda_{n}}(X,p)\), s.t.
\(x_{\omega} \neq p_{\omega}\), is associated with a sequence
\(\{ x_{n} \} \subset X_{p}\), \(\omega\)-a.s., such that
\[\lim_{ \omega} x_{n} = p.\] Furthermore, there exists \(m > 0\) s.t.
\[\frac{m}{\lambda_{n}} \leq d(p,x_{n}).\] This follows from the fact
that by assumption, the sequence \(\lambda_{n}\,d(p,x_{n})\) is
\(\omega\)-a.s. bounded away from \(0\). If \(x_{\omega}\) is also
finite, then for any \(\epsilon > 0\),
\[\frac{(1-\epsilon)\lambda_{n}}{M} \leq d(p,x_{n}) \leq \frac{(1 + \epsilon)\lambda_{n}}{m}. \quad \omega-\text{a.s}.\]
\end{remark}

\subsubsection{Möbius equivalence of asymptotic cones and tangent cones}
\label{sec:mobius_equivalence_of_asymptotic_cones_and_tangent_cones}

One may think of an asymptotic cone of an unbounded quasimetric space as
a tangent cone at the point at infinity. The following Propositions make
this idea precise.

\begin{proposition}
\label{abfe05}
For any sequence \(\{ \lambda_{n} \}\), with
\(\lim_{ \omega }\lambda_{n} = \infty\), the metric Cayley transform at
a finite point \(p \in X\),
\[\mathcal{C}_{p}: (X,d,q,p) \rightarrow (X,d_{p},q,*_{X}),\] induces
the metric Cayley transform
\[\mathcal{C}_{p_{\omega}}: (\operatorname{Tan}^{\omega}_{\lambda_{n}}(X,d),*,p_{\omega}) \to (\operatorname{Cone}^{\omega}_{\lambda_{n}}(X,d_{p}),q_{\omega},*).\]
In particular, \(\operatorname{Tan}^{\omega}_{\lambda_{n}}(X,d)\) and
\(\operatorname{Cone}^{\omega}_{\lambda_{n}}(X,d_{p})\) are Möbius
equivalent.
\end{proposition}

\begin{proof}
Both the asymptotic cone
\((\operatorname{Cone}^{\omega}_{\lambda_{n}}(X,d_{p}),D_{d_{p}/\lambda_{n}})\)
and the tangent cone
\((\operatorname{Tan}^{\omega}_{\lambda_{n}}(X,d),D_{\lambda_{n}d})\)
are quotients of the ultrapower \(X^{\omega}\). We need to show that the
pseudo-quasimetrics \(D_{d_{p}/\lambda_{n}}\) and \(D_{\lambda_{n}d}\)
induce the same equivalence classes on \(X^{\omega}\).

\begin{enumerate}
\def\labelenumi{\arabic{enumi}.}
\tightlist
\item
  \label{da7ea0}
  We first show that for any \(\{ x_{n} \}_{\omega} \in X^{\omega}\),
  \(D_{d_{p}/\lambda_{n}}(\{ q \}_{\omega},\{ x_{n} \}_{\omega}) = \frac{1}{D_{\lambda_{n}d}(\{ p \}_{\omega},\{ x_{n} \}_{\omega})}\).
  If \(q\) is the point at infinity of \((X,d)\), then
  \(d_{p}(q,x_{n}) = \frac{1}{d(p,x_{n})}\) and thus
  \[\begin{split}D_{d_{p}/\lambda_{n}}(\{ q \}_{\omega},\{ x_{n} \}_{\omega}) &= \lim_{  \omega  } \frac{d_{p}(x_{n},q)}{\lambda_{n}} \\&=\lim_{ \omega  } \frac{1}{\lambda_{n}d(p,x_{n})}=\frac{1}{D_{\lambda_{n}d}(\{ p \}_{\omega},\{ x_{n} \}_{\omega})}.\end{split}\]
  If \(q\) is a finite point, then
  \[\left|\frac{1}{\lambda_{n}d(p,x_{n})} - \frac{1}{\lambda_{n}d(q,p)}\right|\leq\frac{d_{p}(q,x_{n})}{\lambda_{n}} \leq \frac{1}{\lambda_{n}d(q,p)} + \frac{1}{\lambda_{n}d(p,x_{n})},\]
  and thus again,
  \(D_{d_{p}/\lambda_{n}}(\{ q \}_{\omega},\{ x_{n} \}_{\omega}) =\frac{1}{D_{\lambda_{n}d}(\{ p \}_{\omega},\{ x_{n} \}_{\omega})}\).
\item
  \label{58176b}
  Let \(\{ x_{n} \}_{\omega}, \{ y_{n} \}_{\omega} \in X^{\omega}\) s.t.
  \(D_{\lambda_{n}d}(\{ p \}_{\omega},\{ x_{n} \}_{\omega}), D_{\lambda_{n}d}(\{ p \}_{\omega},\{ y_{n} \}_{\omega}) \not\in \{ 0,\infty \}\),
  then
  \[\begin{split}D_{d_{p}/\lambda_{n}}(\{ x_{n} \}_{\omega},\{ y_{n} \}_{\omega}) &= \lim_{ \omega } \frac{\lambda_{n}d(x_{n},y_{n})}{\lambda_{n}d(x_{n},p)\,\lambda_{n}d(p,y_{n})} \\&= \frac{D_{\lambda_{n}d}(\{ x_{n} \}_{\omega},\{ y_{n} \}_{\omega})}{D_{\lambda_{n}d}(\{ x_{n} \}_{\omega},\{ p \}_{\omega})D_{\lambda_{n}d}(\{ p \}_{\omega},\{ y_{n} \}_{\omega})}.\end{split}\]
\end{enumerate}

From \ref{da7ea0} and \ref{58176b} follows that
\(\mathcal{C}_{p_{\omega}}\) is a well-defined, bijective map, that maps
the point at infinity in
\(\operatorname{Tan}^{\omega}_{\lambda_{n}}(X,d)\) to the origin
\(q_{\omega}\) in
\(\operatorname{Cone}^{\omega}_{\lambda_{n}}(X,d_{p})\), and the origin
\(p_{\omega}\) in \(\operatorname{Tan}^{\omega}_{\lambda_{n}}(X,d)\) to
the point at infinity in
\(\operatorname{Cone}^{\omega}_{\lambda_{n}}(X,d_{p})\).

Clearly, any sequence \(\{ x_{\omega}^{j} \}\), that converges to a
point \(x_{\omega}\) as \(j \to \infty\) with respect to the quasimetric
\(D_{\lambda_{n}d}\) also converges to \(x_{\omega}\) w.r.t. the
quasimetric \(D_{d_{p}/\lambda_{n}}\) and conversely. Thus
\(\mathcal{C}_{p_{\omega}}\) is a homeomorphism. Since also
\[[x_{\omega},y_{\omega},z_{\omega},w_{\omega}]_{D_{\lambda_{n}d}} = [x_{\omega},y_{\omega},z_{\omega},w_{\omega}]_{D_{d_{p}/\lambda_{n}}},\]
\(\mathcal{C}_{p_{\omega}}\) is a Möbius equivalence.

\end{proof}

\begin{proposition}
\label{b5ff73}
For any sequence \(\{ \lambda_{n} \}\), with
\(\lim_{ \omega }\lambda_{n} = \infty\), the inverse metric Cayley
transform at a finite point \(q \in X\),
\[\mathcal{C}^{q}: (X,d,q,*_{X}) \rightarrow (X,d^{q},q,p),\] induces
the metric Cayley transform
\[\mathcal{C}_{q_{\omega}}: (\operatorname{Cone}^{\omega}_{\lambda_{n}}(X,d),q_{\omega},*) \to (\operatorname{Tan}^{\omega}_{\lambda_{n}}(X,d^{q}),*,p_{\omega}).\]
In particular, \(\operatorname{Cone}^{\omega}_{\lambda_{n}}(X,d)\) and
\(\operatorname{Tan}^{\omega}_{\lambda_{n}}(X,d^{q})\) are Möbius
equivalent.
\end{proposition}

\begin{proof}
Both the asymptotic cone
\((\operatorname{Cone}^{\omega}_{\lambda_{n}}(X,d),D_{d/\lambda_{n}})\)
and the tangent cone
\((\operatorname{Tan}^{\omega}_{\lambda_{n}}(X,d^{q}),D_{\lambda_{n}d^{q}})\)
are quotients of the ultrapower \(X^{\omega}\). We need to show that the
pseudo-quasimetrics \(D_{d/\lambda_{n}}\) and \(D_{\lambda_{n}d^{q}}\)
induce the same equivalence classes on \(X^{\omega}\).

\begin{enumerate}
\def\labelenumi{\arabic{enumi}.}
\tightlist
\item
  \label{73644b}
  We first show that for any \(\{ x_{n} \}_{\omega} \in X^{\omega}\),
  \[D_{\lambda_{n}d^{q}}(\{ p \}_{\omega},\{ x_{n} \}_{\omega}) = \frac{1}{D_{d/\lambda_{n}}(\{ q \}_{\omega},\{ x_{n} \}_{\omega})}.\]
  The point \(p\) is the point at infinity \(*_{X}\) in \((X,d)\), thus
  \(d^{q}(p,x_{n}) = \frac{1}{1 + d(q,x_{n})}\) and thus
  \[\begin{split}D_{\lambda_{n}d^{q}}(\{ p \}_{\omega},\{ x_{n} \}_{\omega}) &= \lim_{\omega}\lambda_{n}\,d^{q}(p,x_{n}) \\&= \lim_{\omega}\frac{\lambda_{n}}{1 + d(q,x_{n})}=\frac{1}{D_{d/\lambda_{n}}(\{ q \}_{\omega},\{ x_{n} \}_{\omega})}.\end{split}\]
\item
  \label{4b2814}
  Let \(\{ x_{n} \}_{\omega}, \{ y_{n} \}_{\omega} \in X^{\omega}\) s.t.
  \(D_{d/\lambda_{n}}(\{ q \}_{\omega},\{ x_{n} \}_{\omega}), D_{d/\lambda_{n}}(\{ q \}_{\omega},\{ y_{n} \}_{\omega}) \not\in \{ 0,\infty \}\),
  then
  \[\begin{split}D_{\lambda_{n}d^{q}}(\{ x_{n} \}_{\omega},\{ y_{n} \}_{\omega}) &= \lim_{ \omega } \frac{\frac{d(x_{n},y_{n})}{\lambda_{n}}}{\left( \frac{1}{\lambda_{n}} + \frac{d(q,x_{n})}{\lambda_{n}} \right)\left( \frac{1}{\lambda_{n}} + \frac{d(p,y_{n})}{\lambda_{n}} \right)} \\&= \frac{D_{d/\lambda_{n}}(\{ x_{n} \}_{\omega},\{ y_{n} \}_{\omega})}{D_{d/\lambda_{n}}(\{ x_{n} \}_{\omega},\{ q \}_{\omega})D_{d/\lambda_{n}}(\{ q \}_{\omega},\{ y_{n} \}_{\omega})}.\end{split}\]
\end{enumerate}

From \ref{73644b} and \ref{4b2814} follows that
\(\mathcal{C}_{q_{\omega}}\) is a well-defined, bijective map, that maps
the point at infinity in
\(\operatorname{Cone}^{\omega}_{\lambda_{n}}(X,d)\) to the origin
\(p_{\omega}\) in
\(\operatorname{Tan}^{\omega}_{\lambda_{n}}(X,d^{q})\), and the origin
\(q_{\omega}\) in \(\operatorname{Cone}^{\omega}_{\lambda_{n}}(X,d)\) to
the point at infinity in
\(\operatorname{Tan}^{\omega}_{\lambda_{n}}(X,d^{q})\).

Clearly, any sequence \(\{ x_{\omega}^{j} \}\), that converges to a
point \(x_{\omega}\) as \(j \to \infty\) with respect to the quasimetric
\(D_{d/\lambda_{n}}\) also converges to \(x_{\omega}\) w.r.t. the
quasimetric \(D_{\lambda_{n}d^{q}}\) and conversely. Thus
\(\mathcal{C}_{q_{\omega}}\) is a homeomorphism. Since also
\[[x_{\omega},y_{\omega},z_{\omega},w_{\omega}]_{D_{d/\lambda_{n}}} = [x_{\omega},y_{\omega},z_{\omega},w_{\omega}]_{D_{\lambda_{n}d^{q}}},\]
\(\mathcal{C}_{q_{\omega}}\) is a Möbius equivalence.

\end{proof}

We will now prove some helpful lemmas that relate asymptotic cones and
tangent cones with \(u\)-separation. Thanks to their Möbius equivalence,
we only prove the lemmas for the case of an asymptotic cone.

\begin{lemma}
\label{e8bbb4}
Let \((X,d,q, *_{X})\) be an unbounded doubly pointed quasimetric space.
Given two distinct points
\(x_{\omega}, y_{\omega} \in \operatorname{Cone}^{\omega}_{\lambda_{n}}(X,q)\),
then for every admissible gauge \(u\) and all sequences \(\{x_n\}\),
\(\{y_n\}\) representing \(x_\omega\) and \(y_\omega\), respectively,
the points \(x_{n}\) and \(y_{n}\) are separated by \(u\) at infinity
for \(\omega\)-almost every \(n\), i.e.
\[d(x_n,y_n) > u(|x_n| + |y_n|),\quad \omega-a.s.\]
\end{lemma}

\begin{proof}
Let us first assume that \(y_{\omega}\) is the point at infinity
\(*_{\omega}\) of \(\operatorname{Cone}^{\omega}_{\lambda_{n}}(X,q)\).
Then \(x_{\omega}\) is necessarily a finite point. This implies that
there exists an \(M>0\) s.t. \[d(q,x_{n}) \leq M \lambda_{n},\] and for
any \(m > M\), \[d(q,y_{n}) \leq m \lambda.\] In order to show
\(d(x_n,y_n) > u(\left|x_n\right|+\left|y_n\right|)\), \(\omega\)-a.s.
we first show two simpler inequalities, namely

\begin{equation}
\label{4441c2}
\left|x_n\right| \leq \frac{\left|y_n\right|}{2}\end{equation}

and

\begin{equation}
\label{d75feb}
\left|x_n\right| \leq 2d(x_n, y_n)\end{equation}

The first inequality follows from
\[|x_{n}| \leq M \lambda_{n} \leq \frac{M}{m} |y_{n}|.\quad \omega-\text{a.s.}\]
Since \(\frac{M}{m}\) can be made arbitrary small the first inequality
follows. For the second inequality, the following holds \(\omega\)-a.s.,
\[
\begin{split}
d(x_{n},y_{n}) &\geq |y_{n}| \left|1 - \frac{|x_{n}|}{|y_{n}|}\right| \\
&\geq |y_{n}|\left(1 - \frac{M}{m}\right).
\end{split}
\] Again since \(\frac{M}{m}\) can be made arbitrary small,
\(|y_{n}|\leq 2\,d(x_{n},y_{n})\).

We now use equation \ref{4441c2} and equation \ref{d75feb} to show that
\(d(x_n,y_n) > u(\left|x_n\right|+\left|y_n\right|)\), \(\omega\)-almost
surely. Indeed, for almost every \(n\), \[
\begin{aligned}
u(|x_{n}| + |y_{n}|) &\leq u(2|y_{n}|) \\
&< \frac{1}{2} |y_{n}| \\
&\leq d(x_{n},y_{n}).
\end{aligned}
\]

Let us now assume that both \(x_{\omega}\) and \(y_{\omega}\) are finite
points, then \(\{x_n\}, \{y_n\} \subset X_{*_{X}}\), \(\omega\)-almost
surely. Assume there exists an admissible gauge \(u\) such that
\[d(x_n,y_n) \leq u(|x_n| + |y_n|), \quad \omega-a.s.\] then for some
\(M \geq 0\), \(\omega-a.s.\),
\[\frac{d(x_n,y_n)}{\lambda_n} \leq \frac{u(|x_n| + |y_n|)}{\lambda_n} \leq \frac{u(2M\lambda_{n})}{\lambda_{n}},\]
as of Remark \ref{374bcc}. So
\[\lim_\omega \frac{d(x_n,y_n)}{\lambda_n} = 0,\] which contradicts the
assumption that \(x_{\omega}\) and \(y_{\omega}\) are distinct.

\end{proof}

\begin{lemma}
\label{6199f1}
Let \((X,d,q,*_{X})\) be an asymptotically chained quasimetric space.
Let \(\{x_n\}\) and \(\{y_n\}\) be two sequences in \(X_{*_{X}}\)
representing the same finite point
\(x_\omega \in \operatorname{Cone}^{\omega}_{\lambda_{n}}(X,q)\) s.t.
\(\lim_{ \omega } d(q,x_{n})=\infty\) and
\(\lim_{ \omega } d(q,y_{n})=\infty\). Then, for any admissible gauge
\(u\), either the points \(x_{n}\) and \(y_{n}\) are separated by \(u\)
\(\omega\)-a.s. at infinity, i.e.
\[x_n,y_n >_{\infty} u, \quad \omega-a.s.\] or there exists a sequence
\(\{w_n\} \subset X_{*_{X}}\) representing the same point \(x_{\omega}\)
s.t \[x_n,w_n >_{\infty} u \quad \omega-a.s.\] and
\[y_n,w_n >_{\infty} u \quad \omega-a.s.\]
\end{lemma}

\begin{proof}
By assumption, for any \(C > 0\), \(|x_{n}| \geq C\) and
\(|y_{n}| \geq C\), \(\omega\)-almost surely. In other words, we can
assume \(|x_{n}|\) and \(|y_{n}|\) to be arbitrary large.

Suppose there exists an admissible gauge \(u\), s.t.
\(d(x_n,y_n) \leq u(|x_{n}| + |y_{n}|)\), for \(\omega\) almost every
\(n\). For \(\omega\)-almost every \(n\) we will construct two annuli,
one around \(y_{n}\) and another around \(x_{n}\) that grow sublinearly
w.r.t. \(\lambda_{n}\) and have nonempty intersection.

By assumption, there exists an admissible gauge \(v_{q}\) such that
\(X\) is asymptotically chained w.r.t. \(v_{q}\).

Define two annuli
\[A_{n} = \{ x \in X \,|\,  2u(4|y_n|) < d(x,y_{n}) \leq  2u(4|y_n|) + v_{q}(|y_n|) \}\]
and
\[A_{n}' = \{ x \in X \, | \, u(4|y_{n}|)< d(x_n,x) \leq 3\, u(4 |y_n|) + v_{q}(|y_n|) \} .\]
Clearly, both annuli grow sublinearly w.r.t \(\lambda_{n}\). We claim
that for \(\omega\)-almost every \(n\) \[
A_{n} \cap A_{n}' \neq \emptyset.
\] Any sequence constructed by picking a point in \(A_{n} \cap A_{n}'\)
for \(\omega\)-almost every \(n\), represents the same point
\(x_{\omega}\) in the asymptotic cone as \(\{ x_{n} \}\) and
\(\{ y_{n} \}\). This uses the fast that \(x_{\omega}\) is a finite
point and Remark \ref{374bcc}.

We first show that \(A_{n} \subset A_{n}'\) for \(\omega\)-almost every
\(n\). Assume that \(A_{n}\) is nonempty. Since we assume
\(d(x_n,y_n) \leq u(|x_{n}| + |y_{n}|)\), \(\omega-a.s.\), we have for
\(\omega\)-almost every \(n\), \[\begin{split}
    2d(x_n,y_n) &\leq 2u(|x_n| + |y_n|) \\
    &\leq |x_n| + |y_n| \\
    &\leq 2|y_n| + d(x_n,y_n). \\
\end{split}\] Thus \[
d(x_{n},y_{n}) \leq 2|y_{n}| \qquad\omega\text{-a.s.}.
\] With \(d(x_n,y_n) \leq u(|x_{n}| + |y_{n}|)\), \(\omega-a.s.\), it
follows,

\begin{equation}
\label{d76a9d}
d(x_n,y_n) \leq u(4|y_n|) \qquad\omega\text{-a.s.}.\end{equation}

For any \(x \in A_{n}\)
\[2u(4 |y_n|)) - d(x_{n},y_{n}) < d(x_n,x) \leq 2 u(4 |y_n|) + d(x_{n},y_{n})+ v_{q}(|y_n|).\]
Therefore thanks to inequality \ref{d76a9d}, \(x \in A_{n}'\).

We claim that \(A_{n}\) is indeed nonempty. \(A_{n}\) is an annulus of
inner radius \(2u(4|y_{n}|)\) and thickness \(v_{q}(|y_{n}|)\). Notice
that for \(\omega\)-almost every \(n\), the origin \(q\) lies beyond the
outer radius of \(A_{n}\), this is because
\(2\,u(4|y_n|) + v_{q}(|y_{n}|) < |y_n|\) \(\omega\)-almost surely.
Hence, any discrete path from \(q\) to \(y_{n}\) passes through
\(A_{n}\). The condition that \(X\) is asymptotically chained w.r.t.
\(v_{q}\), assures that one of the points in the chain
\(x^n_1 = q, \dots, x^n_{k+1} = y_n\) lies in \(A_{n}\). Indeed, since
\[\text{max}_{\, i \in 1,\dots k}\{d(x^n_i,x^n_{i+1})\} < v_{q}(|y_n|),\]
this must be the case for at least one point \(x^n_{i_n}\). Define
\(w_n = x^n_{i_n}\).

It remains to show that \(x_{n}, w_{n} > u\) and \(y_{n},w_{n} > u\) for
\(\omega\)-almost every \(n\). We only show the case of \(\{ x_{n} \}\),
the case for \(\{ y_{n} \}\) follows completely analogously.

Using that \(w_{n} \in A_{n}\) and equation \ref{d76a9d} gives
\[\begin{split}
    u(|x_n| + |w_n|) &\leq u(|x_n| + |y_n| + d(y_n,w_n)) \\
    &\leq u(2|y_n| + d(x_n,y_n) + 2\, u(4|y_n|) + v_{q}(|y_n|)) \\
    &\leq u(2|y_n| + 3\, u(4|y_n|) + v_{q}(|y_n|)) \\
\end{split}\] By sublinearity, for \(\omega\)-almost every \(n\),
\[3\, u(4|y_n|) + v_{q}(|y_n|) \leq 2|y_n|.\] Thus \[
u(|x_n| + |w_n|) \leq u(4|y_n|) < d(x_n,w_n),
\] where the last inequality follows from the fact that
\(w_{n} \in A_{n}'\).

\end{proof}

\subsection{AM-Theorem}
\label{sec:amtheorem}

In this section, we ultimately prove Theorem \ref{85f2fe}. Since tangent
cones and asymptotic cones are Möbius equivalent, it suffices to treat
the case in which we are given a regular AM-map at infinity, \[
f_{u,v} \colon (X,q,*_{X}) \to (Y,q',*_{Y}).
\]

\begin{lemma}
If any regular AM-map \(f_{u,v}: (X,q,*_{X}) \to (Y,q',*_{Y})\) induces
a continuous, injective, quasimöbius map
\[g: \operatorname{Cone}^{\omega}_{\lambda_{n}}(X,q) \to \operatorname{Cone}^{\omega}_{\lambda_{n}'}(Y,q'),\]
then Theorem \ref{85f2fe} follows.
\end{lemma}

\begin{proof}
Assume we are given a general regular AM-map
\(f_{u,v}: (X,q,p) \to (Y,q',p')\). With the help of Example
\ref{a67dd2} we obtain a regular AM-map
\(F_{u,v}: (X,q,*_{X}) \to (Y,q',*_{Y})\), that is at \(*_{X}\), by
possibly postcomposing or precomposing by the Cayley transform or
inverse metric Cayley transform respectively.

By assumption, \(F_{u,v}\) induces a map
\(G: \operatorname{Cone}^{\omega}_{\lambda_{n}}(X,q) \to \operatorname{Cone}^{\omega}_{\lambda_{n}'}(Y,q')\).
The final statement follows from the fact that
\(\operatorname{Tan}^{\omega}_{\lambda_{n}}(X,p)\) and
\(\operatorname{Cone}^{\omega}_{\lambda_{n}}(X,q)\) as well as
\(\operatorname{Cone}^{\omega}_{\lambda_{n}'}(Y,q')\) and
\(\operatorname{Tan}^{\omega}_{\lambda_{n}'}(Y,p')\) are Möbius
equivalent by Proposition \ref{abfe05} and Proposition \ref{b5ff73}.

\end{proof}

\subsubsection{Proof the AM Theorem}
\label{sec:proof_the_am_theorem}

We are given a regular AM-map \(f_{u,v}: (X,q,*_{X}) \to (Y,q',*_{Y})\).
Since \(f_{u,v}\) is continuous at \(*_{X}\) we have for every sequence
\(x_{n} \rightarrow *_{X}\), \(x_{n} \in X_{*_{X}}\), that
\(x_{n}' \rightarrow *_{Y}\). Recall that \(*_{X}\) is an accumulation
point of \(X\). In what follows, we often abbreviate the notation
\(f(x)\) to \(x'\).

Take any sequence \(z_{n} \rightarrow *_{X}\), \(z_{n} \in X_{*_{X}}\).
Define \[\lambda_n = d(q,z_n)\] and \[\lambda_n' = d(q',z_n').\]

By construction and the continuity of \(f\) at \(*_{X}\), the sequences
\(\lambda_{n}\) and \(\lambda_n'\) diverge. We want to show that \(f\)
induces a continuous, injective, quasimöbius map from
\(\operatorname{Cone}^{\omega}_{\lambda_{n}}(X, q)\) to
\(\operatorname{Cone}^{\omega}_{\lambda_{n}'}(Y, q')\).

By construction, the sequence \(\{ z_{n} \}\) induces a point
\(z_{\omega} \in \operatorname{Cone}^{\omega}_{\lambda_{n}}(X, q)\) that
is different from its origin \(q_{\omega}\). In fact,
\(D_{d}(q_{\omega},z_{\omega}) = 1\).

Of all cases, the third case in Theorem \ref{85f2fe} is probably the
easiest to prove. In this case, the quasimöbius condition reduces to the
quasisymmetric condition
\[\frac{d(x',y')}{d(x',z')} \leq \eta\left(\frac{d(x,y)}{d(x,z)}\right).\]

Indeed, assume we have three sequences
\(\{ x_{n} \}, \{ y_{n} \}, \{ z_{n} \} \subset X_{*_{X}}\) that are
pairwise separated by \(u\) at infinity. Then also
\(\{ x_{n} \}, \{ y_{n} \}, \{ z_{n} \}, \{ *_{X} \}\) are pairwise
separated by \(u\) at infinity. Therefore
\[[x_{n}',y_{n}',z_{n}',*_{Y}] \leq \eta([x_{n},y_{n},z_{n},*_{X}]),\]
which by definition of the cross ratio, is actually the three-point
condition
\[\frac{d(x_{n}',y_{n}')}{d(x_{n}',z_{n}')} \leq \eta\left(\frac{d(x_{n},y_{n})}{d(x_{n},z_{n})}\right).\]

\begin{lemma}
\label{1a63dd}
The map \(f_{u,v}\) induces a map \(X^{\omega}\to Y^{\omega}\) between
ultrapowers.

\begin{enumerate}
\def\labelenumi{\arabic{enumi}.}
\tightlist
\item
  If \(\{ x_{n} \}_{\omega} \in X^{\omega}\) represents a finite point
  in \(\operatorname{Cone}^{\omega}_{\lambda_{n}}(X,q)\), then its image
  \(\{ x_{n}' \}_{\omega} \in Y^{\omega}\) represents a finite point in
  \(\operatorname{Cone}^{\omega}_{\lambda_{n}'}(Y,q')\).
\item
  If \(\{ x_{n} \}_{\omega} \in X^{\omega}\) represents the point at
  infinity in \(\operatorname{Cone}^{\omega}_{\lambda_{n}}(X,q)\), then
  its image \(\{ x_{n}' \}_{\omega} \in Y^{\omega}\) represents the
  point at infinity in
  \(\operatorname{Cone}^{\omega}_{\lambda_{n}'}(Y,q')\).
\end{enumerate}
\end{lemma}

\begin{proof}
1.1. If \(\{ x_{n} \}_{\omega}\) represents a finite point \(x_\omega\)
that is different from the origin \(q_{\omega}\) and \(z_{\omega}\),
then by lemma \ref{e8bbb4}, \(\{x_n\}\),\(\{z_n\}\),\(\{q\}\) are
\(\omega\)-a.s. pairwise separated by \(u\).

We may henceforth apply the quasisymmetric condition to this triple,
which yields
\[\frac{d(q',x_n')}{\lambda_n'} = \frac{d(q',x_n')}{d(q',z_n')} < \eta\left(\frac{d(q,x_n)}{d(q,z_n)}\right) = \eta\left(\frac{d(q,x_n)}{\lambda_n}\right)\quad \omega-a.s.\]
Thus \(\frac{d(q',x_n')}{\lambda_n'}\) is bounded and
\(\{ x_{n}' \}_{\omega}\) represents a finite point in
\(\operatorname{Cone}^{\omega}_{\lambda_{n}'}(Y,q')\).

1.2. If \(\{ x_{n} \}_{\omega}\) represents the origin \(q_{\omega}\),
then either \(\lim_{ \omega }d(q,x_{n}) = M < \infty\) or
\(\lim_{ \omega }d(q,x_{n}) = \infty\).

1.2.1. In the first case, \(\{ x_{n} \}\) stays \(\omega\)-a.s. outside
a neighborhood \(V\) of \(*_{X}\), thus \(\{ x_{n}' \}\) stays
\(\omega\)-a.s. outside a neighborhood \(U\) of \(p'\) by the regularity
of \(f_{u,v}\). Hence \(\lim_{ \omega  }d(q',x_{n}') = M' < \infty\) and
\(\{ x_{n}' \}_{\omega}\) represents the origin \(q_{\omega}'\) of
\(\operatorname{Cone}^{\omega}_{\lambda_{n}'}(Y,q')\).

1.2.2. In the second case, \(\{ x_{n} \}\) and \(\{ q \}\) are separated
by \(u\), \(\omega\)-a.s. (see Example \ref{a516c1}). By Lemma
\ref{e8bbb4} also \(\{ x_{n} \}, \{ z_{n} \}, \{ q \}\) are
\(\omega\)-a.s. pairwise separated by \(u\). We may henceforth apply the
quasisymmetric condition to this triple, which yields
\[\frac{d(q',x_n')}{\lambda_n'} = \frac{d(q',x_n')}{d(q',z_n')} < \eta\left(\frac{d(q,x_n)}{d(q,z_n)}\right) = \eta\left(\frac{d(q,x_n)}{\lambda_n}\right)\quad \omega-a.s.\]
Thus \(\lim_{ \omega  }\frac{d(q',x_n')}{\lambda_n'} = 0\) and
\(\{ x_{n}' \}_{\omega}\) represents again the origin \(q_{\omega}'\) of
\(\operatorname{Cone}^{\omega}_{\lambda_{n}'}(Y,q')\).

1.3. If \(\{ x_{n} \}_{\omega}\) represents the point \(z_\omega\), then
by Lemma \ref{6199f1} either \(\{ x_{n} \}, \{ z_{n} \}\) are separated
by \(u\) or there exists a sequence \(\{w_n\}\) representing
\(z_{\omega}\), s.t. \(\{w_n\}, \{x_n\}\) are separated by \(u\) and
\(\{w_n\}\), \(\{z_n\}\) are separated by \(u\).

1.3.1. If \(\{ x_{n} \}, \{ z_{n} \}\) are separated by \(u\), then
\(\{ x_{n} \}, \{ z_{n} \}, \{ q \}\) are \(\omega\)-a.s. pairwise
separated by \(u\) (see Example \ref{a516c1}) and we are back to case 1.

1.3.2. Otherwise, \(\{w_n\}, \{x_n\}, \{q\}\) are
\(\{w_n\}, \{z_n\}, \{q\}\) are pairwise separated by \(u\). We may
henceforth apply the quasisymmetric condition to each of the two triples
separately, which yields

\[\begin{split}
    \frac{d(q',x_n')}{\lambda_n'} &= \frac{d(q',x_n')}{d(q',z_n')} \\
    &= \frac{d(q',x_n')}{d(q',w_n')} \frac{d(q',w_n')}{d(q',z_n')} \\
    &<\eta\left(\frac{d(q,x_n)}{d(q,w_n)}\right)\eta\left(\frac{d(q,w_n)}{d(q,z_n)}\right) \\
    &= \eta\left(\frac{d(q,x_n)}{\lambda_n}\frac{\lambda_n}{d(q,w_n)}\right)\eta\left(\frac{d(q,w_n)}{\lambda_n}\right) \\
    &\rightarrow \eta(1)^2
\end{split}\]

Thus \(\frac{d(q',x_n')}{\lambda_n'}\) is bounded and
\(\{ x_{n}' \}_{\omega}\) represents a finite point in
\(\operatorname{Cone}^{\omega}_{\lambda_{n}'}(Y,q')\).

2.1. If \(\{ x_{n} \}_{\omega}\) represents the point at infinity
\(*_{X}\), then by lemma \ref{e8bbb4}, \(\{x_n\}\),\(\{z_n\}\),\(\{q\}\)
are \(\omega\)-a.s. pairwise separated by \(u\).

We may henceforth apply the quasisymmetric condition to this triple,
which yields
\[\frac{\lambda_n'}{d(q',x_n')} = \frac{d(q',z_n')}{d(q',x_n')} < \eta\left(\frac{d(q,z_n)}{d(q,x_n)}\right) = \eta\left(\frac{\lambda_n}{d(q,x_n)}\right)\quad \omega-a.s.\]
Since \(\lim_{ \omega  } \frac{\lambda_{n}}{d(q,x_{n})} = 0\), also
\(\lim_{ \omega }\frac{\lambda_n'}{d(q',x_n')} 0\), and thus
\(\{ x_{n}' \}_{\omega}\) represents the point at infinity in
\(\operatorname{Cone}^{\omega}_{\lambda_{n}'}(Y,q')\).

\end{proof}

\begin{lemma}
\label{390bb4}
The mapping
\(X^{\omega} \to Y^{\omega}, \{x_n\}_{\omega} \mapsto \{x_n'\}_{\omega}\)
between ultrapowers, induces a well-defined map from the asymptotic cone
\(\operatorname{Cone}^{\omega}_{\lambda_{n}}(X,q)\) to
\(\operatorname{Cone}_{\lambda_{n}'}^{\omega}(Y,q')\).
\end{lemma}

\begin{proof}
1.1. Let \(\{ x_{n} \}_{\omega}\) and \(\{ y_{n} \}_{\omega}\) be two
sequences representing the point at infinity \(*_{X}\) in
\(\operatorname{Cone}_{\lambda_{n}}^{\omega}(X,q)\).

We need to show that then \(\{ x_{n}' \}_{\omega}\) and
\(\{ y_{n}' \}_{\omega}\) represent the same point in
\(\operatorname{Cone}^{\omega}_{\lambda_{n}}(Y,q')\). This follows
immediately from Lemma \ref{1a63dd}, since \(\{ x_{n}' \}_{\omega}\) and
\(\{ y_{n}' \}_{\omega}\) both represent the point at infinity in
\(\operatorname{Cone}^{\omega}_{\lambda_{n}}(Y,q')\).

1.2. Let \(\{x_n\}_{\omega}\) and \(\{ y_{n} \}_{\omega}\) be two
sequences representing the origin \(q_\omega\).

We need to show that then \(\{ x_{n}' \}_{\omega}\) and
\(\{ y_{n}' \}_{\omega}\) represent the same point in
\(\operatorname{Cone}^{\omega}_{\lambda_{n}}(Y,q')\). This follows again
from Lemma \ref{1a63dd}, since \(\{ x_{n}' \}_{\omega}\) and
\(\{ y_{n}' \}_{\omega}\) both represent the origin \(q_{\omega}'\) in
\(\operatorname{Cone}^{\omega}_{\lambda_{n}}(Y,q')\), as follows from
step 1.2 in the proof of Lemma \ref{1a63dd}.

1.3. Let \(\{x_n\}_{\omega}\) and \(\{y_n\}_{\omega}\) be two sequences
representing the same finite point
\(x_{\omega} \in \operatorname{Cone}^{\omega}_{\lambda_{n}}(X,q)\)
different from the origin \(q_{\omega}\).

We need to show that then \(\{ x_{n}' \}_{\omega}\) and
\(\{ y_{n}' \}_{\omega}\) represent the same point in
\(\operatorname{Cone}^{\omega}_{\lambda_{n}}(Y,q')\). By Lemma
\ref{6199f1}, either \(\{ x_{n} \}\) and \(\{ y_{n} \}\) are
\(\omega\)-a.s. separated by \(u\), or there exists a sequence
\(\{w_n\}\) representing the same point \(x_{\omega}\) s.t.
\(\{x_n\},\{w_n\} > u\) and \(\{y_n\},\{w_n\} > u\), \(\omega\)-almost
surely.

1.3.1. In the first case, \(\{x_n\},\{y_n\},\{q\}\) are \(\omega\)-a.s.
pairwise separated by \(u\), as of Lemma \ref{e8bbb4}. We may henceforth
apply the quasisymmetric condition to this triple, which yields
\(\omega\)-a.s. \[\begin{split}
    \frac{d(x_n',y_n')}{\lambda_n'} &= \frac{d(q',x_n')}{\lambda_n'} \frac{d(x_n',y_n')}{d(x_n',q')} \\
    &\leq \frac{d(q',x_n')}{\lambda_n'} \eta \left(\frac{d(x_n,y_n)}{d(x_n,q)}\right) \\
    &= \frac{d(q',x_n')}{\lambda_n'} \eta\left(\frac{d(x_n,y_n)}{\lambda_n} \frac{\lambda_n}{d(x_n,q)}\right).
\end{split}\]

We have that \(\lim_{ \omega  } \frac{\lambda_n}{d(x_n,q)} < \infty\)
since \(x_{\omega} \neq q_{\omega}\) and
\(\lim_{ \omega  }\frac{d(q',x_n')}{\lambda_n'} < \infty\) since
\(\{ x_{n}' \}_{\omega}\) represents a finite point by Lemma
\ref{1a63dd}. Thus
\(\lim_{ \omega } \frac{d(x_n',y_n')}{\lambda_n'} = 0\) and
\(\{ x_{n}' \}_{\omega}\) and \(\{ y_{n}' \}_{\omega}\) represent the
same point in \(\operatorname{Cone}^{\omega}_{\lambda_{n}}(Y,q')\).

1.3.2. In the second case, each triple \(\{x_n\},\{w_n\},\{q\}\) and
\(\{y_n\},\{w_n\},\{q\}\) is \(\omega\)-a.s. pairwise separated by
\(u\). We may henceforth apply the quasisymmetric condition to each of
the two triples separately, which yields \(\omega\)-a.s. \[\begin{split}
    \frac{d(x_n',y_n')}{\lambda_n'} &\leq \frac{d(x_n',w_n')}{\lambda_n'} + \frac{d(w_n',y_n')}{\lambda_n'} \\
    &\leq \frac{d(q',x_n')}{\lambda_n'} \frac{d(x_n',w_n')}{d(x_n',q')} + \frac{d(q',y_n')}{\lambda_n'} \frac{d(y_n',w_n')}{d(y_n',q')} \\
    &\leq \frac{d(q',x_n')}{\lambda_n'} \eta\left(\frac{d(x_n,w_n)}{d(x_n,q)}\right) + \frac{d(q',y_n')}{\lambda_n'} \eta\left(\frac{d(y_n,w_n)}{d(y_n,q)}\right) \\
    &\leq \frac{d(q',x_n')}{\lambda_n'} \eta\left(\frac{d(x_n,w_n)}{\lambda_n}\frac{\lambda_n}{d(x_n,q)}\right) + \frac{d(q',y_n')}{\lambda_n'} \eta\left(\frac{d(y_n,w_n)}{\lambda_n}\frac{\lambda_n}{d(y_n,q)}\right).
\end{split}\]

We again have that
\(\lim_{ \omega  } \frac{\lambda_n}{d(x_n,q)} < \infty\) and
\(\lim_{ \omega  } \frac{\lambda_n}{d(y_n,q)} < \infty\) since
\(x_{\omega} \neq q_{\omega}\) and \(y_{\omega} \neq q_{\omega}\). Also
\(\lim_{ \omega  }\frac{d(q',x_n')}{\lambda_n'} < \infty\) and
\(\lim_{ \omega  }\frac{d(q',x_n')}{\lambda_n'} < \infty\) since
\(\{ x_{n}' \}_{\omega}\) and \(\{ y_{n}' \}_{\omega}\) represent a
finite point by Lemma \ref{1a63dd}. Thus
\(\lim_{ \omega } \frac{d(x_n',y_n')}{\lambda_n'} = 0\) and
\(\{ x_{n}' \}_{\omega}\) and \(\{ y_{n}' \}_{\omega}\) represent the
same point in \(\operatorname{Cone}^{\omega}_{\lambda_{n}}(Y,q')\).

\end{proof}

\begin{proof}[Proof of Theorem \ref{85f2fe}]
By Lemma \ref{390bb4}, the AM-map \(f_{u,v}\) induces a well-defined map
\(g:\operatorname{Cone}^{\omega}_{\lambda_{n}}(X,q) \rightarrow \operatorname{Cone}^{\omega}_{\lambda_{n}'}(Y,q')\).
The map \(g\) is quasimöbius if and only if it is quasisymmetric by
Remark \ref{59a2c8}. We thus show that \(g\) is quasisymmetric and in
particular continuous and injective.

Let \(\{w_n\}, \{x_n\}, \{y_n\}\) be sequences representing three
distinct finite points \(w_\omega,x_\omega\) and \(y_\omega\) in
\(\operatorname{Cone}^{\omega}_{\lambda_{n}}(X,q)\). By lemma
\ref{e8bbb4}, \(\{w_n\}, \{x_n\}, \{y_n\}\) are \(\omega\)-a.s. pairwise
separated by \(u\). So
\[\frac{d(x_n',y_n')}{d(x_n',w_n')} < \eta\left(\frac{d(x_n,y_n)}{d(x_n,w_n)}\right),\]
equivalently
\[\frac{d(x_n',y_n')}{\lambda_n'}\frac{\lambda_n'}{d(x_n',w_n')} < \eta\left(\frac{d(x_n,y_n)}{\lambda_n}\frac{\lambda_n}{d(x_n,w_n)}\right).\]
Taking the \(\omega\) limit, we get
\[\frac{D_{d/\lambda_{n}'}(x_\omega',y_\omega')}{D_{d/\lambda_{n}'}(x_\omega',w_\omega')} \leq \eta\left(\frac{D_{d/\lambda_{n}}(x_\omega,y_\omega)}{D_{d/\lambda_{n}}(x_\omega,w_\omega)}\right).\]
So, if
\(g:\operatorname{Cone}^{\omega}_{\lambda_{n}}(X,q) \setminus \{ *_{\omega} \} \rightarrow \operatorname{Cone}^{\omega}_{\lambda_{n}'}(Y,q') \setminus \{ *_{\omega} \}\)
is not constant, then it is continuous and injective. Remember that by
construction, \(q_{\omega}\) and \(z_{\omega}\) are mapped to two
distinct points \(q'_{\omega}\) and \(z'_{\omega}\), so \(g\) is indeed
not constant. Also \(g\) extends continuously and injectively to
\(\operatorname{Cone}^{\omega}_{\lambda_{n}}(X,q)\), since
\(g(*_{\omega}) = *_{\omega}\), and for any
\(x_{\omega}^{j} \to *_{\omega}\), \(g(x_{\omega}^{j}) \to *_{\omega}\).

\end{proof}

\section{Examples}
\label{sec:examples}

\subsection{Sublinear-Lipschitz equivalences}
\label{sec:sublinearlipschitz_equivalences}

In \cite{Cornulier2019}, Y. Cornulier introduces
\emph{sublinear-Lipschitz maps}.

\begin{definition}
A map \(f: X \rightarrow Y\) between quasimetric spaces is a
\emph{sublinear-Lipschitz map} if there is an admissible gauge
\(u: \mathbb{R}_+ \rightarrow \mathbb{R}\) such that
\[d(f(x),f(y)) \leq Cd(x,y) + C'u(|x| + |y|), \quad \forall x,y \in X,\]
for some constants \(C,C' > 0\).

Two sublinear-Lipschitz maps \(f,f'\) are equivalent if there is an
admissible gauge \(v\) and a constant \(C{''} > 0\) such that
\[d(f(x),f'(x)) \leq C''v(|x|)\] for all \(x \in X\).
\end{definition}

\begin{remark}
\label{80a8bb}
Sublinear-Lipschitz maps between metric spaces form a category. Taking
asymptotic cones, we obtain a functor from the sublinear-Lipschitz
category to the Lipschitz category. The sublinear-Lipschitz category is
in a sense the maximal category with such a property.
\end{remark}

The isomorphisms in the sublinear-Lipschitz category are called
sublinear-Lipschitz equivalences or SBE maps.

\begin{proposition}
\label{7b5b3e}
Every SBE map \(f: X \rightarrow Y\) is an AM-map with linear \(\eta\).
\end{proposition}

\begin{proof}
If \(f\) is \(SBE\), then \(f\) is bi-Lipschitz except at scales below
an admissible gauge \(v\).

Indeed, an \(SBE\) map satisfies
\[c'd(x,y) - C'u(|x| + |y|) \leq d(f(x),f(y)) \leq cd(x,y) + Cu(|x| + |y|),\]
for some gauge \(u\).

Fix the gauge \(v = 2 \frac{C'}{c'}u\), and let \(x,y \in X\) s.t.
\(x,y > v\). Then \[\begin{split}
    d(f(x),f(y)) &\leq cd(x,y) + Cu(|x| + |y|) \\
    &\leq \big(c + \frac{Cc'}{2C'}\big)d(x,y),
\end{split}\] and \[\begin{split}
    d(f(x),f(y)) &\geq c'd(x,y) - C'u(|x| + |y|) \\
    &\geq \frac{c'}{2}d(x,y)
\end{split}\]

So there exists \(D > 0\), s.t. for all \(x,y > v\),
\[\frac{1}{D} d(x,y) \leq d(f(x),f(y)) \leq D d(x,y).\]

In particular, \(f\) is an AM-map with linear \(\eta\).

\end{proof}

\subsection{Assouad-type maps}
\label{sec:assouadtype_maps}

\begin{quote}
\cite{Semmes1996} How can one recognize when a metric space is
bi-Lipschitz equivalent to an Euclidean space?
\end{quote}

Simple as it sounds, this question is not obvious. If a metric space
admits a bi-Lipschitz embedding in \(\mathbb{R}^n\), then it is clearly
doubling. The converse is not true, however; the \(3\)-dimensional
Heisenberg group with the Carnot-Carathéodory metric is doubling, but
does not admit a bi-Lipschitz embedding in \(\mathbb{R}^n\) for any
\(n\).

Doubling in the context of a metric space \((X,d)\) means, that there
exists a constant \(D > 0\) such that for any \(x \in X\) and \(r >0\),
the ball \(B(x,r) = \{y \in X : d(x,y) < r\}\) can be covered by at most
\(D\) balls of radius \(\frac{r}{2}\).

Assouad's embedding theorem \cite{Assouad1983} \cite{Naor2012} states
that any snowflake \(X^\alpha = (X, d^\alpha)\), \(0 < \alpha < 1\), of
a doubling metric space admits a bi-Lipschitz embedding in a Euclidean
space.

It is clear that the Assouad embedding of a doubling metric space into
Euclidean space is an asymptotic-Möbius map. In what follows, we
construct an example of an Assouad mapping from an infinite dimensional
Heisenberg group into a Hilbert space. This construction follows from a
construction of Lee and Naor for the finite dimensional case
\cite{Lee2006}.

Let \(H_{\mathbb{C}}\) be an infinite dimensional complex Hilbert space.
\(H_{\mathbb{C}}\) carries the symplectic form
\(\Omega(a,b) = \text{Im}(\langle a,b\rangle)\).

The \emph{infinite-dimensional Heisenberg group} \(\mathcal{H}_\Omega\),
is the set of tuples \((a,t)\) with \(a \in H_{\mathbb{C}}\),
\(t \in \mathbb{R}\) and the group law
\[(a,t)(a',t') = (a + a', t + t' + 2\,\Omega(a,a')).\]

Let \(G\) be a group with identity element \(e\). A group seminorm on
\(G\) is a function \(G \rightarrow [0,\infty)\) satisfying
\(N(g^{-1}) = N(g)\) for all \(g \in G\), \(N(gh) \leq N(g) + N(h)\) for
all \(g,h \in G\) and \(N(e) = 0\). Moreover, if \(N(g) = 0\) if and
only if \(g = e\), then \(N\) is a group norm on \(G\).

\begin{remark}
\label{f68604}
If \(N_1\) and \(N_2\) are group seminorms, then
\(\sqrt{N_1^2 + N_2^2}\) is a group seminorm.
\end{remark}

The function \[N(a,t) = \sqrt{\sqrt{||a||^4 + t^2} + ||a||^2},\] is a
group norm on the infinite dimensional Heisenberg group
\(\mathcal{H}_\Omega\). In this case,
\(N_1 = (||a||^4 + t^2)^{\frac{1}{4}}\) is the Korányi norm.

Given \(N\), \[d_N((a,t),(a',t')) = N((a,t)(a',t')^{-1})\] is a right
invariant metric on \(\mathcal{H}_\Omega\). We want to show that
\(d_{N}\) is a kernel conditionally of negative type on
\(\mathcal{H}_\Omega\).

\begin{definition}
A continuous real valued kernel \(K\) on a topological space \(X\) is
\emph{conditionally of negative type}, if:

\begin{enumerate}
\def\labelenumi{\arabic{enumi}.}
\tightlist
\item
  \(K(x, x)=0\) for all \(x\) in \(X\);
\item
  \(K(y, x)=K(x, y)\) for all \(x, y\) in \(X\);
\item
  for any \(n\) in \(\mathbb{N}\), any elements \(x_1, \ldots, x_n\) in
  \(X\), and any real numbers \(c_1, \ldots, c_n\) with
  \(\sum_{i=1}^n c_i=0\), the following inequality holds:
  \[\sum_{i=1}^n \sum_{j=1}^n c_i c_j K\left(x_i, x_j\right) \leq 0\]
\end{enumerate}
\end{definition}

If \(d_N\) is conditionally of negative type, then by the GNS
construction there is a real Hilbert space \(H\) and an isometry
\(T:(\mathcal{H}_\Omega,\sqrt{d_N}) \rightarrow H\). In particular, the
mapping \(T:(\mathcal{H}_\Omega,d_N) \rightarrow H\) is an AM-map.

\begin{theorem}[GNS construction {\cite{Gelfand1943}\cite{Segal1947}}]
Let \(K\) be a continuous kernel on a topological space \(X\), that is
conditionally of negative type, and let \(x_0 \in X\). Then there exists
a real Hilbert space \(H\) and a continuous mapping
\(T: X \rightarrow H\) with the following properties:

\begin{enumerate}
\def\labelenumi{\arabic{enumi}.}
\tightlist
\item
  \(K(x, y)=\|T(x)-T(y)\|^2\) for all \(x\), \(y\) in \(X\);
\item
  the linear span of \(\left\{T(x)-T\left(x_0\right): x \in X\right\}\)
  is dense in \(H\).
\end{enumerate}

Moreover, the pair \((H, T)\) is unique, up to a canonical isomorphism.
That is, if \((H', T')\) is another pair satisfying 1. and 2., then
there exists a unique affine isometry \(L: H \rightarrow H'\) such that
\(T'(x)=L(T(x))\) for all \(x \in X\).
\end{theorem}

\begin{definition}
A \emph{function of positive type} on a topological group \(G\) is a
continuous function \(\Phi: G \to \mathbb{C}\) such that the kernel on
\(G\) defined by
\[\left(g_1, g_2\right) \mapsto \Phi\left(g_2^{-1} g_1\right)\] is of
positive type, that is,
\[\sum_{i=1}^n \sum_{j=1}^n c_i \overline{c_j} \,\Phi\left(g_j^{-1} g_i\right) \geq 0\]
for all \(g_1, \ldots, g_n \in G\) and
\(c_1, \ldots, c_n \in \mathbb{C}\).
\end{definition}

\begin{remark}
\label{e872ce}
For some authors, kernels of positive type are
\emph{\href{https://en.wikipedia.org/wiki/Positive-definite_kernel}{positive
definite Hermitian kernels}}, as for example in
\cite[Chapter~8]{Benyamini2000}. We avoid this terminology, since the
reader may expect positive definiteness to include the additional
condition that
\(\sum_{i=1}^n \sum_{j=1}^n c_i \overline{c_j} \,K\left(x_i, x_j\right) = 0\)
if and only if \(c_1 = \cdots = c_n = 0\).
\end{remark}

\begin{remark}[{\cite[Proposition~C.1.6]{Bekka2008}}]
\label{d45467}
Let \(K_1\) and \(K_2\) be kernels of positive type on a topological
space \(X\). Then \[K_1 K_2:(x, y) \mapsto K_1(x, y) K_2(x, y)\] is a
kernel of positive type.
\end{remark}

\begin{lemma}
\label{368f85}
For any \(\lambda \in \mathbb{R}\), the function
\[\Phi_\lambda(a,t) = e^{-|\lambda| \|a\|^2 + i\lambda t}\] on
\(\mathcal{H}_\Omega\) is of positive type.
\end{lemma}

\begin{proof}
Indeed, the function \(\Phi_\lambda\) induces the kernel
\[K((a,s),(b,t)) := \exp(-|\lambda|\|a-b\|^2 + i\lambda (s-t-2\,\Omega(a,b))).\]

The kernel can be rewritten as a product of exponentials
\[\exp(-|\lambda|(\|a\|^2 + \|b\|^2) + i\lambda(s-t))\exp(2|\lambda| (\operatorname{Re}(\langle a,b\rangle) - i\operatorname{sign}(\lambda)\, \Omega(a,b)) ).\]

By Remark \ref{d45467}, the product of kernels of positive type is a
kernel of positive type.

The first factor is of positive type, since for all
\(c_1,\dots,c_n \in \mathbb{C}\), \[
\begin{split}
\sum_{i}^{n}\sum_{j}^{n} &c_{i}\bar{c}_{j} \exp(-|\lambda|(\|a_{i}\|^2 + \|a_{j}\|^2) + i\lambda(t_{i}-t_{j})) \\
&= \sum_{i}^{n}\sum_{j}^{n} c_{i} \exp(-|\lambda|\|a_{i}\|^2 + i\lambda t_{i})  \overline{\left(c_{j} \exp(-|\lambda|\|a_{j}\|^2 + i\lambda t_{j})\right)} \\
&= \left|\sum_{i}^{n} c_{i} \exp(-|\lambda|\|a_{i}\|^2 + i\lambda t_{i})  \right|^{2} \geq 0.
\end{split}
\]

The exponential of a kernel of positive type is again a kernel of
positive type \cite[Proposition~8.2.]{Benyamini1999}. Thus the kernel
\(K\) is of positive type, if
\[\operatorname{Re}(\langle a,b \rangle) - i\operatorname{sign}(\lambda)\,\Omega(a,b)\]
is of positive type. The above equals the Hermitian form
\(\langle a,b \rangle\) if \(\lambda < 0\) and
\(\overline{\langle a,b \rangle}\) if \(\lambda \geq 0\) and thus
clearly it is a kernel of positive type.

\end{proof}

\begin{definition}
Let \(G\) be a topological group. A continuous function
\(\Psi: G \rightarrow \mathbb{R}\) is \emph{conditionally of negative
type} if the kernel \(K\) on \(G\), defined by
\(K(g, h)=\Psi\left(h^{-1} g\right)\) for \(g, h \in G\), is
conditionally of negative type.
\end{definition}

\begin{remark}[{\cite[Proposition~C.2.4]{Bekka2008}}]
\label{5498b7}
Let \(G\) be a topological group whose identity element we denote by
\(e\).

\begin{enumerate}
\def\labelenumi{\arabic{enumi}.}
\tightlist
\item
  Let \(\{ \Psi_t \}_t\) be a family of functions conditionally of
  negative type on \(G\) converging pointwise on \(G\) to a continuous
  function \(\Psi: G \rightarrow \mathbb{R}\). Then \(\Psi\) is a
  function conditionally of negative type.
\item
  Let \(\Phi\) be a real valued function of positive type on \(G\). Then
  \[g \mapsto \Psi(e) - \Psi(g)\] is a function conditionally of
  negative type.
\end{enumerate}
\end{remark}

\begin{proposition}
\label{49a80e}
The metric
\(d_N: \mathcal{H}_\Omega \times \mathcal{H}_\Omega \rightarrow \mathbb{R}\)
is conditionally of negative type.
\end{proposition}

\begin{proof}
The existence of rotation-invariant \(\frac{1}{2}\)-stable distributions
implies, that for all \(\epsilon > 0\), there exists a non-negative
integrable function
\(\varphi_\epsilon: \mathbb{R} \rightarrow [0,\infty)\) with Fourier
transform \(\hat{\varphi}_\epsilon(t) = e^{-\epsilon\sqrt{|t|}}\). Note
that
\(\frac{1}{2\pi} \hat{\hat{\varphi}}_\epsilon(x) = \varphi_\epsilon(x)\).
For reference see for example
\cite[Proposition~2.5~(xii)~and~Theorem~14.14]{Sato1999}.

By Lemma \ref{368f85},
\[F_\epsilon(a,t) = \frac{1}{2\pi}\int_{\mathbb{R}} e^{-|\lambda| \|a\|^2 + i\lambda t} \varphi_\epsilon(\lambda)\,d\lambda\]
is a function of positive type on \(\mathcal{H}_\Omega\). Notice that it
is also real valued. We make use of the Cauchy probability density
\[h_k(x) = \frac{k}{\pi}\frac{1}{k^2 + x^2}\] with scale parameter
\(k > 0\). It's characteristic function is
\(\hat{h}_{k}(t) = e^{-k|t|}\), \cite[Example~2.11]{Sato1999}.

With this at hand, we write \[\begin{split}
    F_\epsilon(a,t) &= \frac{1}{2\pi} \int_{\mathbb{R}} e^{i\lambda t}\hat{h}_{\|a\|^2}(\lambda)\hat{\hat{\varphi}}_\epsilon(\lambda)\,d\lambda \\
    &= \frac{1}{2\pi} \int_{\mathbb{R}} e^{i\lambda t} \left(\widehat{h_{\|a\|^2} * \hat{\varphi}_\epsilon}\right)(\lambda)\,d\lambda \\
    &= (h_{\|a\|^2} * \hat{\varphi}_\epsilon)(t).
\end{split}\]

For every \(\epsilon\), \(F_\epsilon\) is of positive type and real
valued, by Remark \ref{5498b7}.2, \(\frac{1 - F_\epsilon}{\epsilon}\) is
conditionally of negative type.

By Remark \ref{5498b7}.1, also its pointwise limit is conditionally of
negative type. \[\begin{split}
    \lim_{\epsilon \rightarrow 0} \frac{1 - F_\epsilon(a,t)}{\epsilon} &= \lim_{\epsilon \rightarrow 0} \big[h_{\|a\|^2} * \frac{1 - \hat{\varphi}_\epsilon}{\epsilon}\big] \\
    &= \lim_{\epsilon \rightarrow 0} \int_{\mathbb{R}} \frac{1 - e^{-\epsilon\sqrt{|x|}}}{\epsilon}h_{\|a\|^2}(t-x)dx \\
    &= \frac{\|a\|^2}{\pi} \int_\mathbb{R}\frac{\sqrt{|x|}}{\|a\|^4 + (t-x)^2} dx. \\
\end{split}\]

We now show that for all \(r,t \in \mathbb{R}\)
\[r^2 \int_\mathbb{R} \frac{\sqrt{|x|}}{r^4 + (t-x)^2} dx= \pi\sqrt{\sqrt{r^4 + t^2} + r^2}.\]

By changing variables \(x=r^2y\) and \(s=t/r^2\), the left hand side can
be written as
\[\int_0^\infty \Big(\frac{1}{1 + (s-y)^2} + \frac{1}{1 + (s+y)^2}\Big)\sqrt{y}\,dy\]

This integral equals
\[\lim_{r \rightarrow 0}\lim_{R \rightarrow \infty} \frac{1}{2}\int_{C_{r,R}} \Big(\frac{1}{1 + (s-z)^2} + \frac{1}{1 + (s+z)^2}\Big)\sqrt{z}\,dz, \]
where \(C_{r,R}\) is the keyhole contour with a branch cut along the
positive real axis. The above integrand has simple poles at \(i \pm s\)
and \(-i \pm s\).

We compute the respective residues. \[
\begin{split}
\operatorname{Res}\!\left(\frac{\sqrt{z}}{1 + (s - z)^2},\, i + s\right) &= \frac{\sqrt{i + s}}{2i}, \\
\operatorname{Res}\!\left(\frac{\sqrt{z}}{1 + (s - z)^2},\, -i + s\right) &= \frac{\sqrt{-i + s}}{-2i}, \\
\operatorname{Res}\!\left(\frac{\sqrt{z}}{1 + (s + z)^2},\, i - s\right) &= \frac{\sqrt{i - s}}{2i}, \\
\operatorname{Res}\!\left(\frac{\sqrt{z}}{1 + (s + z)^2},\, -i - s\right) &= \frac{\sqrt{-i - s}}{-2i}.
\end{split}
\]

By Cauchy's residue theorem, the integral equals
\[\frac{\pi}{2}\,(\sqrt{i+s} - \sqrt{-i+s} + \sqrt{i-s} - \sqrt{-i-s}).\]

This further simplifies to \[\begin{split}
    \pi\,\operatorname{Re}(\sqrt{i+s} + \sqrt{i-s}) &= \pi\,(\operatorname{Re}(\sqrt{i+s}) + \operatorname{Im}(\sqrt{i+s})) \\
    &= \pi\,\Bigg(\sqrt{\frac{\sqrt{1 + s^2} + s}{2}} + \sqrt{\frac{\sqrt{1 + s^2} - s}{2}}\Bigg) \\
    &= \pi\sqrt{\sqrt{1 + s^2} + 1}. \\
\end{split}\]

\end{proof}

\section{Applications to dimension theory}
\label{sec:applications_to_dimension_theory}

\subsection{Dimension theory of finitely-generated groups}
\label{sec:dimension_theory_of_finitelygenerated_groups}

In 1993 M. Gromov introduced the notion of asymptotic dimension as a
large scale analogue of Lebesgue's covering dimension \cite{Gromov1993}.
The asymptotic dimension of a finitely generated group is a
quasi-isometric invariant. Its most prominent application goes back to
Guoliang Yu, who showed that any finitely generated group with finite
homotopy type and finite asymptotic dimension satisfies the Novikov
conjecture \cite{Yu1998}.

\begin{definition}
The \emph{covering dimension} (or topological dimension) of a
topological space \(X\) is the minimum number \(n\) such that every
finite open cover \(\mathcal{V}\) of \(X\) has an open refinement
\(\mathcal{U}\) with order \(n+1\). We write
\(\operatorname{dim}X = n\). If no such minimal \(n\) exists, the space
is said to have infinite covering dimension.
\end{definition}

\begin{remark}
\label{4d4381}
In the literature there are three common definitions of topological
dimension: the covering dimension (\(\operatorname{dim}\)), the large
inductive dimension (\(\operatorname{Ind}\)) and the small inductive
dimension (\(\operatorname{ind}\)). All three notions agree for
separable metric spaces \cite[p.~220]{Engelking1995}. On nonseparable
metric spaces, still \(\operatorname{Ind} X=\operatorname{dim} X\) by
the Katétov-Morita theorem \cite[p.~218]{Engelking1995}, however
\(\operatorname{ind} X\) can be different.
\end{remark}

\begin{remark}[Countable sums {\cite[p.~221]{Engelking1995}}]
\label{879747}
If a metrizable space \(X\) can be represented as the union of a
sequence \(F_1, F_2, \ldots\) of closed subspaces such that
\(\operatorname{dim} F_i \leq n\) for \(i= 1,2, \ldots\), then
\(\operatorname{dim} X \leq n\).
\end{remark}

\begin{remark}[Dimension lowering mappings {\cite[p.~242]{Engelking1995}}]
\label{4b8034}
If \(f: X \rightarrow Y\) is a closed map of a metrizable space \(X\) to
a metrizable space \(Y\) and there exists an integer \(k \geq 0\) such
that \(\operatorname{dim} f^{-1}(y) \leq k\) for every \(y \in Y\), then
\(\operatorname{dim} X \leq \operatorname{dim} Y + k\).
\end{remark}

\begin{definition}
Let \(X\) be a metric space. The \emph{asymptotic dimension} of \(X\) is
the minimal number \(n\) such that for every uniformly bounded open
cover \(\mathcal{V}\) of \(X\) there exists a uniformly bounded open
cover \(\mathcal{U}\) of \(X\) of order \(n+1\) such that
\(\mathcal{V}\) refines \(\mathcal{U}\). We write
\(\text{asdim} X = n\). If no such minimal \(n\) exists, the space is
said to have infinite asymptotic dimension.
\end{definition}

\begin{remark}
\label{e11255}
An open cover \(\mathcal{V}\) of \(X\) is \emph{uniformly bounded}, if
there exists a \(C > 0\), s.t.
\[\sup_{V \in \mathcal{V}} \sup_{x,y \in V} d(x,y) \leq C.\]
\end{remark}

\begin{remark}
\label{efb8dd}
The asymptotic dimension differs from the classical covering dimension
(or topological dimension) only by the additional assumption, that open
covers need to be uniformly bounded.
\end{remark}

\begin{remark}
\label{dd97a5}
The asymptotic dimension of a finitely generated group \(\Gamma\) has
several interesting implications. For example,
\(\operatorname{asdim}\Gamma = 0\) if and only if \(\Gamma\) is finite
(Proposition 65. in \cite{Bell2008}) and
\(\operatorname{asdim}\Gamma = 1\) if and only if \(\Gamma\) is
virtually free (Theorem 66. in \cite{Bell2008}).
\end{remark}

\begin{remark}[{\cite[Theorem~3.5]{Carlsson2004}}]
\label{f140a6}
For a simply connected nilpotent Lie group \(N\) with the left-invariant
Riemannian metric, \[\operatorname{asdim}(N)=\operatorname{dim}(N) .\]
\end{remark}

If \(\Gamma\) is a finitely generated virtually nilpotent group equipped
with the word norm, then its asymptotic cone \(G_\infty\) is a Carnot
group equipped with a left-invariant Carnot-Carathéodory distance
\(d_{CC}\) \cite{Pansu1983}. Similarly, if \((G,d_{g})\) is a simply
connected nilpotent Lie group, equipped with some left-invariant
Riemannian metric \(d_{g}\), then again its asymptotic cone \(G_\infty\)
is a Carnot group equipped with a left-invariant Carnot-Carathéodory
distance \(d_{CC}\) \cite{Pansu1983}.

\begin{remark}
\label{9a490c}
Algebraically, the Lie group \(G_{\infty}\) is constructed as follows:
by assumption, there exists in \(\Gamma\) a nilpotent subgroup \(N\) of
finite index; let \[\Gamma^{\prime}=N / \operatorname{tor} N\] then
\(\Gamma^{\prime}\) is torsion-free; according to Malcev
\cite{Malcev1951}, \(\Gamma^{\prime}\) is isomorphic to a discrete
cocompact subgroup of a nilpotent Lie group \(G\); then \(G_{\infty}\)
is the graded group associated with \(G\), i.e.~the simply connected Lie
group whose Lie algebra is the graded algebra \(\mathfrak{g}_{\infty}\)
associated with \(\mathfrak{g}\).
\end{remark}

\begin{remark}
\label{640c46}
A Carnot group with a left-invariant Carnot-Carathéodory distance
\(d_{CC}\) is proper since \(d_{CC}\) induces the usual
manifold-topology \cite{Sussmann1973} and any left-invariant metric on a
locally compact homogeneous space, admitting proper dilations, is
proper. In particular, \(d_{CC}\) is complete.
\end{remark}

\begin{remark}
\label{7ea47d}
When one varies the norm on the horizontal bundle, the
Carnot-Carathéodory distance remains bi-Lipschitz equivalent to itself.
Therefore, one can for example speak of Lipschitz mappings between
Carnot groups without referring to a particular Carnot-Carathéodory
distance.
\end{remark}

\begin{remark}
\label{6242b0}
For any finitely generated polycyclic group \(\Gamma\) with the word
norm, its asymptotic dimension \(\operatorname{asdim}\Gamma\) equals its
Hirsch length \(h(\Gamma)\) \cite[Theorem~66]{Bell2008}. For a finitely
generated nilpotent group \(\Gamma\) with the word norm, Pansu's
asymptotic cone Theorem \cite{Pansu1983} says that the covering
dimension of its asymptotic cone \(G_{\infty}\) equals the Hirsch length
\(h(\Gamma)\). Thus for a finitely generated nilpotent group,
\[\operatorname{asdim}\Gamma = \operatorname{dim} G_{\infty}.\]
\end{remark}

\begin{proof}[Proof of Theorem \ref{0c2f8e}]
The mapping \(f\) induces a quasisymmetric mapping \(g\) between the
asymptotic cones \(G_\infty\) and \(G_\infty'\) of \(\Gamma\) and
\(\Gamma'\), respectively.

In particular, \(g\) is an injective mapping from \(G_\infty\) to
\(G'_\infty\). Let \(\bar{B}_{n}\) be a sequence of closed balls in
\(G_{\infty}\), s.t. \(\bigcup_{n}^{\infty} \bar{B}_{n} = G_{\infty}\).
All closed balls \(\bar{B}_{n}\) are compact by Remark \ref{640c46}.
Thus, all \(g|_{\bar{B}_{n}}: \bar{B}_{n} \to G_{\infty}'\) are closed.
By Remark \ref{4b8034} for all \(n\),
\[\operatorname{dim} \bar{B}_{n} \leq \operatorname{dim} G_{\infty}'.\]
Thus
\(\operatorname{dim} G_{\infty} \leq \operatorname{dim} G_{\infty}'\) by
Remark \ref{879747}. Since \(\Gamma\) and \(\Gamma'\) are nilpotent,
Remark \ref{6242b0} implies
\[\operatorname{asdim}(\Gamma) = \operatorname{dim}(G_\infty) \leq \operatorname{dim}(G'_\infty) = \operatorname{asdim}(\Gamma').\]
The dimension increase
\[\operatorname{dim}(G) \leq \operatorname{dim}(G')\] for simply
connected nilpotent Lie groups \(G\) and \(G'\), follows from the above
together with Remark \ref{f140a6}.

If \(\operatorname{asdim}(\Gamma) = \operatorname{asdim}(\Gamma')\), the
invariance of domain theorem implies that
\(g: G_{\infty} \to g(G_{\infty})\) is a homeomorphism. A homeomorphism
is quasisymmetric iff its inverse is quasisymmetric.

Since both \(g\) and \(g^{-1}\) are quasisymmetric, they are a.e.
differentiable by the Pansu-Rademacher-Stepanov theorem, and the
differential of \(g\) is a group isomorphism intertwining the dilations
of \(G_{\infty}\) and \(G_{\infty}'\) \cite{Pansu1989}.

The converse follows from Cornulier
\cite[Proposition~2.9]{Cornulier2011} saying that the asymptotic cones
\(G_{\infty}\) and \(G'_{\infty}\) are bi-Lipschitz equivalent if and
only if there exists an SBE map between \(\Gamma\) and \(\Gamma'\). By
the previous paragraph, there exists an isomorphism between
\(G_{\infty}\) and \(G_{\infty}'\) that induces a graded isomorphism of
Lie algebras. Remark \ref{7ea47d} implies that \(G_{\infty}\) and
\(G_{\infty}'\) equipped with their respective Carnot-Carathéodory
distance are bi-Lipschitz equivalent. In particular, there exists an
asymptotic-Möbius map between them.

\end{proof}

\subsection{Dimension theory of CAT(0)-spaces}
\label{sec:dimension_theory_of_cat0spaces}

In the late 1990s, Bruce Kleiner introduced the notion of geometric
dimension \cite{Kleiner1999} to settle some open questions posed by
Gromov on the asymptotic geometry of \(\operatorname{CAT}(0)\)-spaces
\cite[pp.127]{Gromov1993}. In \cite{Gromov1993}, Gromov presents several
definitions of the rank of a \(\operatorname{CAT}(0)\)-space \(X\) and
conjectures their equivalence under suitable assumptions. All
definitions share the idea that \(X\) behaves hyperbolically in the
dimensions above \(\operatorname{rank}X\). (Gromov refers to this as the
guiding principles in the asymptotic geometry of
\(\operatorname{CAT}(0)\)-spaces.) Kleiner showed that most of the
definitions agree for complete \(\operatorname{CAT}(0)\)-spaces,
provided one replaces the topological dimension used by Gromov with the
geometric dimension.

\begin{definition}
The \emph{geometric dimension} of a \(\operatorname{ CAT }(0)\)-space
\(X\) is the maximal topological dimension of all compact subsets in
\(X\). We write \(\operatorname{geom-dim}(X)\) for the geometric
dimension of \(X\).
\end{definition}

\begin{example}
\label{9397da}
A \(0\)-dimensional \(\operatorname{ CAT }(0)\)-space is a single point.
\(1\)-dimensional \(\operatorname{ CAT }(0)\)-spaces coincide with
\(\mathbb{R}\)-trees \cite{Kleiner1999}.
\end{example}

\begin{remark}
\label{832771}
Since the topological dimension is increasing under inclusion
\cite[p.~220]{Engelking1995}, clearly
\[\operatorname{geom-dim}(X) \leq \operatorname{dim}(X).\]
\end{remark}

\begin{remark}
\label{174821}
The geometric dimension is particularly well-adapted to the theory of
\(\operatorname{ CAT }(0)\)-spaces. For example for any locally compact
Hadamard space on which some \(\Gamma \leq \operatorname{Isom}(X)\) acts
cocompactly,
\[\operatorname{geom-dim}(X) = 1 + \operatorname{geom-dim}(\partial_{T}X),\]
where \(\partial_{T}X\) is the Tits boundary
\cite[Theorem~C]{Kleiner1999}. The geometric dimension proves useful
even for \(\operatorname{ CAT }(0)\)-spaces that are not proper
\cite{Caprace2009}.
\end{remark}

\begin{remark}
\label{2fc8b6}
It is conjectured that for \(\operatorname{CAT}(k)\)-spaces, the
topological dimension agrees with the geometric dimension
\cite[p.133]{Gromov1993}. Kleiner shows the equivalence for all
separable \(\operatorname{CAT}(k)\)-spaces \cite{Kleiner1999}.
\end{remark}

\begin{definition}
A metric space \(X\) whose asymptotic cones are
\(\operatorname{CAT}(0)\)-spaces has \emph{telescopic dimension}
\(\leq n\) if every asymptotic cone has geometric dimension \(\leq n\).
We write \(\operatorname{tele-dim}(X)\) for the telescopic dimension of
\(X\).
\end{definition}

\begin{example}
\label{5d075d}
It follows from \cite[p.37]{Gromov1993} and Example \ref{9397da}, that
Gromov-hyperbolic spaces are exactly those metric spaces, whose
telescopic dimension is equal to \(1\).
\end{example}

\begin{remark}
\label{00bfa8}
The telescopic dimension is essentially Gromov's first definition for
the \(\operatorname{rank}\) of a \(\operatorname{CAT}(0)\)-space \(X\),
provided \(X\) admits a cocompact isometric action of some group
\(\Gamma\) and one replaces Gromov's topological dimension with the
geometric dimension \cite[p.~127]{Gromov1993}. Indeed, if \(X\) is a
symmetric space of noncompact type without Euclidean factors, then
\[\operatorname{tele-dim}(X)=\operatorname{rank}(X),\] see
\cite[Theorem~5.2.1]{Kleiner1997}.
\end{remark}

\begin{remark}
\label{d30a5c}
For any \(\operatorname{CAT}(\kappa)\)-space \(X\), its geometric
dimension bounds its telescopic dimension \cite[Lemma~11.1]{Lytchak2005}
i.e. \[\operatorname{tele-dim}(X) \leq \operatorname{geom-dim}(X).\]
\end{remark}

\begin{remark}
\label{53166d}
Any asymptotic cone of a Euclidean building \(X\) of rank \(r\) is a
Euclidean building of rank \(r\) \cite[Theorem~5.1.1]{Kleiner1997}. In
particular it is a \(\operatorname{CAT}(0)\)-space
\cite[Theorem~10A.4]{Bridson1999}, and
\[\operatorname{tele-dim}(X) = \operatorname{geom-dim}(X) = r.\]
\end{remark}

\begin{remark}
\label{rank}
When it comes to the rank of a building, there is an unfortunate clash
of terminology in the literature. The algebraic \emph{rank} of a
building \(\Delta\) is defined as the maximal rank of simplices
\(\sigma \in \Delta\). The geometric \emph{rank} of a building is the
topological dimension of its apartments in a geometric realization
\(|\Delta|\). When talking about \emph{the rank} of a building we always
mean its geometric rank and we avoid the notion of algebraic rank
altogether.
\end{remark}

\begin{proof}[Proof of Theorem \ref{2f4976}]
The mapping \(f\) induces a quasisymmetric (in particular injective)
mapping \(g\) between the asymptotic cones
\(\operatorname{Cone}^\omega_{\lambda_{n}} (X,o)\) and
\(\operatorname{Cone}^\omega_{\lambda_{n}'} (Y,o')\) as constructed in
the proof of Theorem \ref{85f2fe}. Both are Hadamard spaces and
henceforth proper.

Let \(\bar{B}_{n}\) be a sequence of closed balls in
\(\operatorname{Cone}^\omega_{\lambda_{n}} (X,o)\), s.t.
\(\bigcup_{n}^{\infty} \bar{B}_{n} = \operatorname{Cone}^\omega_{\lambda_{n}} (X,o)\).
All closed balls \(\bar{B}_{n}\) are compact. Thus, all
\(g|_{\bar{B}_{n}}: \bar{B}_{n} \to \operatorname{Cone}^\omega_{\lambda_{n}'} (Y,o')\)
are closed. By Remark \ref{4b8034} for all \(n\), \[
\operatorname{dim} \bar{B}_{n} \leq \operatorname{dim} g(\bar{B}_{n}) \leq \operatorname{geom-dim} \operatorname{Cone}^\omega_{\lambda_{n}'} (Y,o').
\] Thus
\(\operatorname{dim} \operatorname{Cone}^\omega_{\lambda_{n}} (X,o) \leq \operatorname{geom-dim} \operatorname{Cone}^\omega_{\lambda_{n}'} (Y,o')\)
by Remark \ref{879747}. Together with Remark \ref{832771}
\[\operatorname{geom-dim} \operatorname{Cone}^\omega_{\lambda_{n}} (X,o) \leq \operatorname{geom-dim} \operatorname{Cone}^\omega_{\lambda_{n}'} (Y,o').\]
As of Remark \ref{977137}, any asymptotic cone of \((X,q,*_{X})\) is of
the form as constructed in the proof of Theorem \ref{85f2fe}. Thus
\[\operatorname{tele-dim}(X) \leq \operatorname{geom-dim} \operatorname{Cone}^\omega_{\lambda_{n}'} (Y,o') \leq \operatorname{tele-dim}(Y).\]

\end{proof}

Theorem \ref{2f4976}, Remark \ref{53166d} and Example \ref{5d075d} imply
the following corollaries.

\begin{corollary}
There does not exist an AM-map from a Euclidean building of finite rank
\(r\) to a Euclidean building of lower rank.
\end{corollary}

\begin{corollary}
If there exist an AM-map from a \(\operatorname{ CAT }(0)\)-space to a
hyperbolic metric space, then \(X\) is also a hyperbolic metric space.
\end{corollary}

\section{On metric cotype obstructions}
\label{sec:on_metric_cotype_obstructions}

In the early 1970s, the notions of (Rademacher) type and cotype arose in
the study of Banach spaces, through the foundational work of Enflo,
Maurey, and Pisier \cite{Pisier1974} \cite{Maurey1976}. Their
introduction provided quantitative tools for assessing how closely a
Banach space resembles a Hilbert space at finite scales.

This led to the development of what became known as the \emph{local
theory of Banach spaces} --- the study of finite-dimensional structures
inside general Banach spaces, especially how these structures behave
under renormings, embeddings, and ultraproduct limits.

Ribe's theorem \cite{Ribe1976} suggested that such ``local properties''
admit a purely metric formulation. This insight sparked significant
interest, especially in the 1990s and 2000s, in extending the local
theory to general metric spaces. Notable developments include the work
of Bourgain--Milman--Wolfson and, later, Mendel--Naor \cite{Mendel2008},
who introduced metric cotype and established nonlinear analogues of the
classical local theory.

\begin{definition}[Metric cotype{}]
Let \((X, d)\) be a metric space and \(q>0\). The space \((X, d)\) is
said to have \emph{metric cotype} \(q\) with constant \(C\) if for every
integer \(n \in \mathbb{N}\), there exists an even integer \(m\), such
that for every \(f: \mathbb{Z}_m^n \rightarrow X\),
\[\sum_{j=1}^n \mathbb{E}_x\left[d\left(f\left(x+\frac{m}{2} e_j\right), f(x)\right)^q\right] \leq C^q m^q \mathbb{E}_{\varepsilon, x}\left[d(f(x+\varepsilon), f(x))^q\right],\]
where the expectations above are taken with respect to uniformly chosen
\(x \in \mathbb{Z}_m^n\) and \(\varepsilon \in\{-1,0,1\}^n\)
(\(\left\{e_j\right\}_{j=1}^n\) denotes the standard basis of
\(\mathbb{R}^n\)).
\end{definition}

\begin{definition}
A metric space \((X,d)\) has \emph{Enflo type} \(p\) if there exists
\(T \in(0, \infty)\) so that for all \(n \geq 1\) and every function
\(f:\{-1,1\}^n \rightarrow X\),
\[\begin{array}{r}\mathbb{E}_{\varepsilon}\, d(f(\varepsilon), f(-\varepsilon))^p \leq T^p \sum_{j=1}^n \mathbb{E}_{\varepsilon} \,d\left(f\left(\varepsilon_1, \ldots, \varepsilon_{j-1}, \varepsilon_j, \varepsilon_{j+1}, \ldots, \varepsilon_n\right),\right. \\ \left.f\left(\varepsilon_1, \ldots, \varepsilon_{j-1},-\varepsilon_j, \varepsilon_{j+1}, \ldots, \varepsilon_n\right)\right)^p .\end{array}\]
\end{definition}

\begin{remark}[{\cite[Theorem~1.2]{Mendel2008}}]
\label{d025fe}
Let \(X\) be a Banach space, and \(q \in[2, \infty)\). Then \(X\) has
metric cotype \(q\) if and only if \(X\) has Rademacher cotype \(q\).
\end{remark}

\begin{remark}[{\cite{Ivanisvili2020}}]
\label{8402ca}
Let \(X\) be a Banach space, and \(p \in [1, 2]\). Then \(X\) has Enflo
type \(p\) if and only if \(X\) has Rademacher type \(p\).
\end{remark}

\begin{example}
\label{e6146b}
\(L_p\) has (Rademacher) type \(\min \{p, 2\}\) and Rademacher cotype
\(\max \{p, 2\}\) \cite[Theorem~6.2.14]{Albiac2016}.
\end{example}

\begin{definition}
\label{nontrivial}
A metric space \(X\) has \emph{nontrivial metric cotype}, if
\[q_X=\inf \{q \geqslant 2: X \text { has metric cotype } q\} < \infty.\]
\end{definition}

\begin{definition}
\label{nontrivial}
A metric space \(X\) has \emph{nontrivial type}, if
\[p_X=\sup \{p \geqslant 1: X \text { has Enflo type } p\} > 1.\]
\end{definition}

\begin{remark}
\label{Maurey}
Let \(X\) be an infinite-dimensional Banach space and suppose
\(p_X=\sup \{p: X \text{ has type } p\}\) and
\(q_X=\inf \{q: X\text{ has cotype }q\}\). Then the Maurey-Pisier
theorem says that both \(\ell_{p_X}\) and \(\ell_{q_{X}}\) are finitely
representable in \(X\) \cite[pp.~85]{Milman1986}.
\end{remark}

\begin{remark}
\label{sharp}
The ``scaling parameter'' \(m\) in the definition of metric cotype
depends on \(X\), \(q\) and the dimension \(n\). For any non-singleton
metric space \((X, d)\) with metric cotype \(q\),
\[m(X,q,n) \geqslant \frac{1}{C} n^{1 / q},\] as shown in
\cite[Lemma~2.3]{Mendel2008}. A non-singleton metric space \((X, d)\)
has \emph{sharp metric cotype} \(q\) if it has metric cotype \(q\) with
\(m, C\) satisfying the bounds \(m \lesssim_{q, X} n^{1 / q}\) and
\(C \lesssim_{q, X} 1\) (expressed in Vinogradov notation, Definition
\ref{314608}).
\end{remark}

\begin{example}
\label{c1972c}
Every Banach space of nontrivial type and Rademacher cotype \(q\) has
sharp metric cotype \(q\). The question if this is true for all Banach
spaces remains open \cite{Mendel2008}.
\end{example}

\begin{example}
\label{61981d}
q-barycentric metric spaces have sharp metric cotype \(q\)
\cite[Theorem~5]{Eskenazis2019}.
\end{example}

\begin{example}
\label{ff2aa3}
Any Alexandrov space of nonpositive curvature is \(2\)-barycentric (see
e.g.~\cite[Lemma~4.1]{Lang2000} or \cite[Theorem~6.3]{Sturm2003}) and
thus has sharp metric cotype \(2\).
\end{example}

\begin{remark}
\label{41d4bb}
Naor observed, that the Rademacher cotype increases under quasisymmetric
maps \cite{Naor2012_2}. More precisely, if \(Y\) is a Banach space with
nontrivial type and \(X\) embeds quasisymmetrically into \(Y\), then
\[q_X \le q_Y.\] The analogous statement for metric cotype remains
unknown.
\end{remark}

Nevertheless, Naor's proof allows the target space \(Y\) to be a metric
space of sharp metric cotype. We may utilize this flexibility to apply
this result to asymptotic cones.

\begin{remark}
\label{15602f}
Notice that if \((X,d)\) has metric cotype \(q\) with constant \(C\) and
scaling parameter \(m\), then any rescaled version has metric cotype
\(q\) with the same constant \(C\) and scaling parameter \(m\). Since
the metric cotype condition involves only finite sums/averages of
nonnegative real numbers, it is thus stable under passing to asymptotic
cones or tangent cones.
\end{remark}

\begin{proof}[Proof of Theorem \ref{7f9774}]
Assume that there exists a metric space \(X\) of sharp metric cotype
\(q < p\) and an AM-map \(f: \ell_{p} \to X\). Then together with
Example \ref{1554d9}, \(f\) induces a quasisymmetric embedding
\(g: L_{p}(\mu) \to Z\) between \(L_{p}(\mu)\) and some asymptotic cone
\(Z\) of \(X\).

By Remark \ref{15602f}, \(Z\) has sharp metric cotype \(q\). Also
\[q_{L_{p}(\mu)} = p,\] by Example \ref{e6146b}. The rest of the proof
follows from \cite{Naor2012_2}.

\end{proof}

Example \ref{c1972c} and Example \ref{ff2aa3} imply the following
corollaries.

\begin{corollary}
Suppose that \(p,q \in [2,\infty]\) satisfy \(p > q\). Then there does
not exist a regular AM-map from \(\ell_{p}\) to \(\ell_{q}\).
\end{corollary}

\begin{corollary}
Suppose \(p > 2\), then there does not exist a regular AM-map from
\(\ell_{p}\) to any Alexandrov space of nonpositive curvature.
\end{corollary}

\section{Appendix}
\label{sec:appendix}

\begin{definition}[Vinogradov notation{}]
\label{314608}
For two quantities \(X\) and \(Y\), we write \(|X| \lesssim |Y|\)
(equivalently, \(|Y| \gtrsim |X|\)) to indicate that there exists a
universal constant \(C>0\) such that \(|X| \leq C\,|Y|\). The notation
\(|X| \asymp |Y|\) means that both inequalities \(|X| \lesssim |Y|\) and
\(|Y| \lesssim |X|\) hold simultaneously.

When the implicit constant is allowed to depend on additional
parameters, this dependence is indicated by a subscript. For instance,
given auxiliary objects such as \(q\) or a space \(M\), the notation
\(|X| \lesssim_{q,M} |Y|\) signifies that \(|X| \leq C(q,M)\,|Y|\) for
some constant \(C(q,M)>0\) depending only on \(q\) and \(M\). The
analogous interpretations apply to \(|X| \gtrsim_{q,M} |Y|\) and
\(|X| \asymp_{q,M} |Y|\).
\end{definition}

\begin{lemma}
\label{ae93d0}
If \(X\) is a Banach space, then all asymptotic cones
\(\operatorname{Cone}^\omega_{\lambda_{n}} (X,0)\) of \(X\) are linearly
isometric to the ultrapower \(X^{\omega}\) of \(X\).
\end{lemma}

\begin{proof}
To any point
\(\left[x_{n}\right]_\omega \in \operatorname{Cone}^\omega_{\lambda_{n}} (X,0)\),
we associate the point
\(\left\{  \frac{x_{n}}{\lambda_{n}}  \right\}_{\omega} \in X^{\omega}\).
This map
\[\phi: \left[x_{n}\right]_\omega \mapsto \left\{  \frac{x_{n}}{\lambda_{n}}  \right\}_{\omega},\]
is well-defined since if \(\{ x_{n} \}\) and \(\{ y_{n} \}\) define the
same point in \(\operatorname{Cone}^\omega_{\lambda_{n}} (X,0)\), then
\(\lim_{ \omega } \| \phi(\left[x_{n}\right]_\omega) - \phi(\left[y_{n}\right]_\omega)\| = \lim_{ \omega } \| \frac{x_{n} - y_{n}}{\lambda_{n}} \| = 0\).
It is equally norm-preserving and linear, and thus a linear isometric
embedding. One easily checks, that
\[\phi^{-1}:  \left\{  x_{n}  \right\}_{\omega} \mapsto \left[\lambda_{n} x_{n}\right]_\omega,\]
is the inverse of \(\phi\).

\end{proof}

\begin{lemma}
\label{0c35a0}
If \((X,q,*_{X})\) is asymptotically chained, then any asymptotic cone
of \(X\) at \(q\) is isometric to an asymptotic cone
\(\operatorname{Cone}^\omega_{\lambda_{n}} (X,q)\), where
\(\lambda_{n} = d(q,z_{n})\) for some sequence \(z_{n} \in X\).
\end{lemma}

\begin{proof}
Notice that the Lemma follows, once we can show that for any given
ultrafilter \(\omega\) and any sequence \(\lambda_{n}\) with
\(\lim_{ _\omega} \lambda_{n} = \infty\), there exists a sequence
\(z_{n} \in X\), s.t.
\[\lim_{ \omega } \frac{\lambda_{n}}{|z_{n}|} = 1.\]

We may assume that \(\lambda_{n} > 0\), and let
\(z_n \in X_{*_{X}} \backslash B_{\lambda_n}(o)\) be
\(\varepsilon\)-minimal, in the sense that for any other
\(x \in X_{*_{X}} \backslash B_{\lambda_n}(o)\) with
\(d(o, x) \leqslant d\left(o, z_n\right)\),
\(\left|d(o, x)-d\left(o, z_n\right)\right|<\varepsilon\). Such a
\(z_{n}\) always exists, since \(*_{X}\) is assumed to be not isolated.

Since \(X\) is asymptotically chained, there exists a sequence
\[o = x_{0}, x_1, \ldots, x_{k-1}, x_k = z_{n} \quad \text{s.t.} \quad d\left(x_i, x_{i+1}\right) \leq v\left(\left|z_n\right|\right).\]
Let \(x_i\) be the maximal point inside \(B_{\lambda_{n}}(o)\), in the
sense that for all \(j>i, x_j \not\in B_{\lambda_{n}}(o)\).

We may assume that \(x_{i}=x_{k-1}\). Indeed, if this is not the case,
then either \(d\left(o, x_{i+1}\right) \leq d\left(o, z_n\right)\) and
hence
\(\left|d\left(o, x_{i+1}\right)-d\left(o, z_n\right)\right|<\varepsilon\),
so
\[d\left(x_j, x_{j+1}\right) \leqslant v\left(\left|z_{n}\right|\right) \leqslant v\left(\left|x_{i+1}\right|+\varepsilon\right),\]
or \(d\left(o, x_{i+1}\right)>d(o,z_{n})\) and hence
\[d\left(x_j, x_{j+1}\right) \leqslant v\left(|z_{n}|\right) \leqslant v\left(\left|x_{i+1}\right|\right).\]
Except for an additive \(\varepsilon\)-error that is uniform over \(n\),
we may replace \(z_{n}\) by \(x_{i + 1}\).

We now relate \(\lambda_{n}\) with \(|z_{n}|\) via the triangle
inequality,
\[\left|z_n\right| \leq \left|x_{k-1}\right|+v\left(\left|z_n\right|\right) \leq \lambda_n + v\left(\left|z_n\right|\right).\]
In particular \[
1-\frac{v\left(\left|z_n\right|\right)}{|z_{n}|} \leq \frac{\lambda_n}{|z_{n}|} \leq 1.
\]

\end{proof}

\bibliographystyle{alpha}
\bibliography{Asymptotic-Moebius_maps}
\end{document}